\documentclass[11pt,twoside]{article}

\usepackage{mathrsfs,amsfonts,amsmath,amssymb}
\usepackage{latexsym}
\usepackage{comment} 
\newtheorem{satz}{Theorem}

\newtheorem{theorem}[satz]{Theorem}
\newtheorem{lemma}[satz]{Lemma}

\newtheorem{corollary}[satz]{Corollary}
\newtheorem{remark}[satz]{Remark}

\def\Z{\mathbb {Z}}
\def\F{\mathbb {F}}
\def\E{\mathsf{E}}

\def\a{\alpha}

\def\d{\delta}

\def\({\big (}
\def\){\big )}

\def\G{\Gamma}

\def\le{\leqslant}
\def\ge{\geqslant}
\def\_phi{\varphi}
\def\eps{\varepsilon}

\def\Gr{{\mathbf G}}
\def\FF{\widehat}

\def\la{\lambda}
\def\D{\Delta}

\def\tr{\mathrm{tr}}

\def\SL{{\rm SL}}

\def\Aff{{\rm Aff}}

\newcommand{\zk}{\ensuremath{\mathfrak{k}}}

\author{Shkredov I.D.}
\title{Growth in Chevalley groups relatively to parabolic subgroups  and some applications
\footnote{This work is supported by the Russian Science Foundation under grant 19--11--00001.}
}
\date{}
\begin{document}
	\maketitle


\begin{center}
	Annotation.
\end{center}

{\it \small
	Given a Chevalley group $\Gr(q)$ and a parabolic subgroup $P\subset \Gr(q)$, we prove that for any set $A$
	there is a certain growth of $A$ relatively to $P$, namely, either $AP$ or $PA$ is much larger than $A$. 
	Also, we study a question about intersection of $A^n$ with parabolic subgroups $P$ for large $n$.  
	We apply our method 
	to obtain some results on a modular form of Zaremba's conjecture from the theory of continued fractions and make the first step towards Hensley's conjecture about some Cantor sets with Hausdorff dimension greater than $1/2$.
}
\\

\section{Introduction}
\label{sec:introduction}

In this paper we study some aspects of growth in Chevalley groups.
Developing the ideas from \cite{H} it was proved in \cite{BGT}, \cite{PS} that any finite simple group of Lie type has growth in the following sense.

\begin{theorem}
	Let $\Gr$ be a finite simple group of Lie type with rank $r$ and $A$ be a generating subset of $\Gr$. 
	Then either $A^{3}=\Gr$ or
$$
\left|A^{3}\right|>|A|^{1+c} \,, 
$$
where $c>0$ depends only on $r$.
\\
In particular, there is $n\ll (\log |\Gr|/\log |A|)^{C(r)}$ such that $A^n = \Gr$.  
\label{t:growth_in_Lie}
\end{theorem}
Theorem above gives an affirmative answer to the well--known Babai's conjecture \cite{Babai} for 
finite simple groups 
$\Gr$ having  
bounded rank. 
In this paper we consider two variants of this problem for Chevalley groups $\Gr(q)$
defined over the field $\F_q$. 
The motivation both of our problems goes back to a question from Number Theory, see \cite{MS_Zaremba} and Section \ref{sec:application}. 
Let us describe the first problem.
Let $P\subseteq \Gr(q)$ be any parabolic subgroup of $\Gr(q)$. 
First of all, what can we say about size of the product  of an arbitrary set $A\subseteq  \Gr(q)$ by $P$?
Of course, $A$ can be a family of cosets of $P$, say, $x_1P, \dots, x_k P$ and thus $AP$ does not grow. 
Similarly, if $A = \bigsqcup_j P y_j$, then $PA=A$. 
Nevertheless, we show that $A$ must grow either after left multiplication or after right multiplication.   
It reminds the sum--product phenomenon, see, e.g., \cite{TV} and indeed our new application to continued fractions (see Section \ref{sec:application} below) is connected with this area, see the discussion of the main results in \cite{NG_S}.

Let us formulate our
first theorem 
in the simplified form (actually, the restriction $A\cap P = \emptyset$ can be relaxed hugely, see Theorem \ref{t:growth_in_P} from Section \ref{sec:proof}). 
Our regime throughout this paper: $q$ tends to infinity  and rank is fixed.

\begin{theorem}
	Let $\Gr (q)$ be a Chevalley group and $P \subset \Gr (q)$ be a parabolic subgroup. 
	Then for any set $A\subseteq \Gr (q)$ with $A\cap P = \emptyset$ one has
	\begin{eqnarray}\label{f:growth_in_P_intr}
	\max\{ |AP|, |PA| \} \ge \frac{\sqrt{|A| |P|q}}{2} \,.
	\end{eqnarray} 
	\label{t:growth_in_P_intr}
\end{theorem}

For example, if $|A| \le  |P|$, then $\max\{ |AP|, |PA| \} \gg |A| \sqrt{q}$ and this is larger than $|A|$.

Theorem above helps us to study the second problem. 
Let $A$ be an arbitrary subset of a group $\Gr$ and $\G$ be a subgroup of $\Gr$. 
Can we guarantee that for a certain reasonable $n$ (say, $n$ depends on $\log |\Gr|/\log |A|$ only)
one has $A^n \cap \G \neq \emptyset$? 
The representation theory (see \cite{SX}, \cite{Gowers_quasirandom}, \cite{LS_representations} or Theorem \ref{t:A^n_cap_B} below) allows to show that 
{\it any} set $A\subset \Gr (q)$ of size at least $\Gr (q) q^{-r+\d}$, where $r$ is rank of $\Gr (q)$ and  $\d >0$ is an arbitrary real number effectively generates the whole group $\Gr (q)$.
In particular,  $A^n \cap \G \neq \emptyset$ for $n \ll_r \d^{-1}$ (see Section \ref{sec:representation}) and this bound is essentially sharp. 
We show that if one wants to find a non--trivial intersection with any parabolic subgroup of $\Gr (q)$,  then 
it is possible to 
break this barrier.

\begin{theorem}
	Let $q$ be an odd number, 
	$\Gr(q)$ be a 
	Chevalley group,  $P \subset \Gr (q)$ be a parabolic subgroup and $P_*$ be a proper parabolic subgroup of the maximal size.    
	Suppose that $|A| \ge  |P_*| q^{-1+\d}$, where $\d >0$ is a real number.
	Then there is $n$,  $n \ll_r \d^{-1}$ such that $A^{n} \cap P \neq \emptyset$. 
	\label{t:A^n_cap_B_intr}
\end{theorem}


It turns out that the method of the proof of Theorems \ref{t:growth_in_P_intr}, \ref{t:A^n_cap_B_intr} has some applications to the theory of continued fractions, namely, to Zaremba's conjecture. 
Let us recall the formulation.
Let $a$ and $q$ be two positive coprime integers, $0<a<q$. 
By the Euclidean algorithm, a rational $a/q$ can be uniquely represented as a regular continued fraction
\begin{equation}\label{exe}
\frac{a}{q}=[0;b_1,\dots,b_s]=
\cfrac{1}{b_1 +\cfrac{1}{b_2 +\cfrac{1}{b_3+\cdots +\cfrac{1}{b_s}}}}
~,\qquad b_s \ge 2.
\end{equation}
Zaremba's famous conjecture \cite{zaremba1972methode} posits that there is an absolute constant $\zk$ with the following property:
for any positive integer $q$ there exists $a$ coprime to  $q$ such that in the continued fraction expansion (\ref{exe}) all partial quotients are bounded:
\[
b_j (a) \le \zk,\,\, 1\le j  \le s = s(a).
\]
In fact, Zaremba conjectured that $\zk=5$.
For large prime $q$, even $ \zk=2$ should be enough, as conjectured by Hensley \cite{hensley_SL2}, \cite{hensley1996}.
This theme is rather popular especially at the last time, see, e.g., \cite{hensley1992continued}, \cite{Kontorovich_survey} or short surveys about this area in \cite{NG_S}, \cite{MS_Zaremba}. 
We just mention a result of  Korobov \cite{Korobov_book} who proved that one can always take growing $\zk$, namely, $\zk  = O(\log q)$ for prime $q$ (such result is also true for composite $q$).

In \cite{MS_Zaremba} we have proved a "modular"\, version of Zaremba's conjecture.

\begin{theorem}
	There is an absolute constant $\zk$ such that for any prime number  $p$ 
	there exist some positive integers $q = O(p^{30})$,  $q\equiv 0 \pmod p$ and  $a$, $a$ coprime with $q$ having  the property that the ratio $a/q$ has partial quotients  bounded by  $\zk$.  
	\label{t:main_MS}
\end{theorem}

The first theorem  in this direction was  
proved by Hensley 
in \cite{hensley_SL2} and after that in  \cite{MOW_MIX}, \cite{MOW_Schottky}. 
Now using results similar to Theorems \ref{t:growth_in_P_intr}, \ref{t:A^n_cap_B_intr} above  and, of course,  growth results in $\SL_2 (\F_p)$ of Helfgott \cite{H}, we improve Theorem \ref{t:main_MS}.

\begin{theorem}
	Let $\epsilon \in (0,1]$ be any real number.
	There is a constant $\zk = \zk (\epsilon)$ 
	such that for any prime number  $p$ 
	there exist some positive integers 
	$q=O(p^{1+\epsilon})$, 
	$q\equiv 0 \pmod p$ and  $a$, $a$ coprime with $q$ having  the property that the ratio $a/q$ has partial quotients  bounded by  $\zk$.  
	\label{t:main_MS_new}
\end{theorem}


Clearly, Theorem \ref{t:main_MS_new} is the best possible up to $\eps$ and it is the limit of our method.

Another result on continued fractions (see Theorem \ref{t:CF_alphabet} from Section \ref{sec:application}) 
is even more interesting than Theorem  \ref{t:main_MS_new} because its generality 
and because it is the first (weak) confirmation of Hensley's hypothesis  \cite[Conjecture 3]{hensley1996}.
Namely, let now the partial quotients $b_j$ belong to a finite set $\mathcal{A} \subset \mathbb{N}$, $|\mathcal{A}|\ge 2$ and  suppose that 
the Hausdorff dimension of the correspondent Cantor set is strictly greater than $1/2$ (all the definitions are contained in Section \ref{sec:application}).
Then we show that a full analogue of Theorem \ref{t:main_MS_new} takes place (with other constants, of course).

We finish the Introduction posing a weak version of Babai's conjecture.
Even for sufficiently large subgroups $\G$ the answer to our question is non--obvious.  

\bigskip 

{\bf Problem.} 
{\it 
	Let $\Gr$ be a finite simple non--abelian group, $\G\subset \Gr$  be a subgroup and $A\subseteq \Gr$ be an arbitrary (generating) set. 
	Is it true that $A^n  \cap \G \neq \emptyset$ with $n\ll (\log |\Gr|/\log |A|)^{C}$, where  $C>0$ is an absolute constant? 
}

\bigskip

If $A=A^{-1}$, then the set $AA = AA^{-1}$ obviously contains the unit element and hence 
the answer to the problem is trivially 
affirmative 
(moreover if $|A||\G| > |\Gr|$, then the Dirichlet principle shows that $|AA^{-1} \cap \G|>1$ and hence we can find a non--trivial element in $AA^{-1}$).
Thus  we cannot assume that $A=A^{-1}$ and, actually, this restriction is very important for some applications as for our modular version of Zaremba's conjecture.

We thank Nikolai Vavilov, Misha Rudnev for  useful discussions and Nikolay Moshchevitin for valuable discussions and  encouragement.

\section{Definitions} 
\label{sec:definitions}

Let $\Gr$ be a group with the identity $1$.
Given two sets $A,B\subset \Gr$, define  the \textit{product set}  of $A$ and $B$ as 
$$AB:=\{ab ~:~ a\in{A},\,b\in{B}\}\,.$$
In a similar way we define the higher product sets, e.g., $A^3$ is $AAA$. 
Let $A^{-1} := \{a^{-1} ~:~ a\in A \}$. 
As usual, having two subsets $A,B$ of a group $\Gr$,  denote by 
\[
\E(A,B) = |\{ (a,a_1,b,b_1) \in A^2 \times B^2 ~:~ a^{-1} b = a^{-1}_1 b_1 \}| 
\]
the {\it common energy} of $A$ and $B$. 
Clearly, $\E(A,B) = \E(B,A)$ and by the Cauchy--Schwarz inequality 
\begin{equation}\label{f:energy_CS}
\E(A,B) |A^{-1} B| \ge |A|^2 |B|^2 \,.
\end{equation}
We  use representation function notations like $r_{AB} (x)$ or $r_{AB^{-1}} (x)$, which counts the number of ways $x \in \Gr$ can be expressed as a product $ab$ or  $ab^{-1}$ with $a\in A$, $b\in B$, respectively. 
For example, $|A| = r_{AA^{-1}}(1)$ and  $\E (A,B) = r_{AA^{-1}BB^{-1}}(1) =\sum_x r^2_{A^{-1}B} (x)$. 
In this paper we use the same letter to denote a set $A\subseteq \Gr$ and  its characteristic function $A: \Gr \to \{0,1 \}$. 
We write $\F^*_q$ for $\F_q \setminus \{0\}$, where $q=p^s$, $p$ is a prime number, and $(a_1,\dots,a_l)$ for the greatest common divisor of some  given positive integers $a_1, \dots, a_l$.  
If $m$ divides $n$, then we write $m|n$.

Let $g\in \Gr$ and let  $A \subseteq \Gr$ be any set. 
Then  put $A^g = g A g^{-1}$ and, similarly, let $x^g := g x g^{-1}$, where  $x\in \Gr$. 
We write $N(A)$ for the normalizer of a set $A$, that is, $N(A) = \{ g\in \Gr ~:~ A^g = A\}$. 
If $H \subseteq \Gr$ is a subgroup, then we 
use the notation 
$H\le \Gr$.

In the paper we consider the group $\SL_2 (\F_q)$  of matrices 
\[
g=
\left( {\begin{array}{cc}
	a & b \\
	c & d \\
	\end{array} } \right) = (ab|cd) \,, \quad \quad a,b,c,d\in \F_q \,, \quad \quad ad-bc=1 \,, 
\]
as well as other {\it classical} groups as 
${\rm PSL}_n (q)$, ${\rm SU}_n (q)$,  ${\rm Sp}_n (q)$, ${\rm \Omega^{\eps}_n} (q)$ and so on.
Also, we use the usual Lie notation ${\rm A}_n (q)$, ${\rm B}_n (q)$ and so on.

The signs $\ll$ and $\gg$ are the usual Vinogradov symbols.
All logarithms are to base $2$.

\section{Simple facts from  the representation theory}
\label{sec:representation}

First of all, we recall some notions and simple facts from the representation theory, see, e.g., \cite{Naimark} or \cite{Serr_representations}.
For a finite group $\Gr$ let $\FF{\Gr}$ be the set of all irreducible unitary representations of $\Gr$. 
It is well--known that size of $\FF{\Gr}$ coincides with  the number of all conjugate classes of $\Gr$.  
For $\rho \in \FF{\Gr}$ denote by $d_\rho$ the dimension of this representation. 
By $d_{\min} (\Gr)$ denote 
the quantity 
$\min_{\rho \neq 1} d_\rho$.  
We write $\langle \cdot, \cdot \rangle$ for the corresponding  Hilbert--Schmidt scalar product 
$\langle A, B \rangle = \langle A, B \rangle_{HS}:= \tr (AB^*)$, where $A,B$ are any two matrices of the same sizes. 
Put $\| A\| = \sqrt{\langle A, A \rangle}$.
Clearly, $\langle \rho(g) A, \rho(g) B \rangle = \langle A, B \rangle$ and $\langle AX, Y\rangle = \langle X, A^* Y\rangle$.
Also, we have $\sum_{\rho \in \FF{\Gr}} d^2_\rho = |\Gr|$.

For any function $f:\Gr \to \mathbb{C}$ and $\rho \in \FF{\Gr}$ define the matrix $\FF{f} (\rho)$, which is called the Fourier transform of $f$ at $\rho$ by the formula 
\begin{equation}\label{f:Fourier_representations}
\FF{f} (\rho) = \sum_{g\in \Gr} f(g) \rho (g) \,.
\end{equation}
Then the inverse formula takes place
\begin{equation}\label{f:inverse_representations}
f(g) = \frac{1}{|\Gr|} \sum_{\rho \in \FF{\Gr}} d_\rho \langle \FF{f} (\rho), \rho (g^{-1}) \rangle \,,
\end{equation}
and the Parseval identity is 
\begin{equation}\label{f:Parseval_representations}
\sum_{g\in \Gr} |f(g)|^2 = \frac{1}{|\Gr|} \sum_{\rho \in \FF{\Gr}} d_\rho \| \FF{f} (\rho) \|^2 \,.
\end{equation}
The main property of the Fourier transform is the convolution formula 
\begin{equation}\label{f:convolution_representations}
\FF{f*g} (\rho) = \FF{f} (\rho) \FF{g} (\rho) \,,
\end{equation}
where the convolution of two functions $f,g : \Gr \to \mathbb{C}$ is defined as 
\[
(f*g) (x) = \sum_{y\in \Gr} f(y) g(y^{-1}x) \,.
\]
Finally, it is easy to check that for any matrices $A,B$ one has $\| AB\| \le \| A\|_{o} \| B\|$ and $\| A\|_{o} \le \| A \|$, where  the operator $l^2$--norm  $\| A\|_{o}$ is just the absolute value of the maximal singular value of $A$.  
In particular, it shows that $\| \cdot \|$ is indeed a matrix norm.

\bigskip

For any function $f : \Gr \to \mathbb{C}$ consider the  Wiener norm of $f$ defined as 
\begin{equation}\label{def:Wiener}	
\| f\|_W := \frac{1}{|\Gr|} \sum_{\rho \in \FF{\Gr}} d_\rho \| \FF{f} (\rho) \| \,. 
\end{equation}

\begin{lemma}
	Let $\G \le \Gr$. 
	Then $\| \G \|_W \le 1$.  
\label{l:G_Wiener}
\end{lemma} 
\begin{proof}
	Since $\G$ is a subgroup, we see using  \eqref{f:Parseval_representations} twice that 
	\[
	|\G|^2 = |\{\gamma_1 \gamma_2 = \gamma_3 ~:~ \gamma_1,\gamma_2,\gamma_3 \in \G\}| = \frac{1}{|\Gr|} \sum_{\rho \in \FF{\Gr}} d_\rho \langle \FF{\G}^2 (\rho), \FF{\G} (\rho) \rangle 
	\le
	\]
	\[ 
	\le
	\frac{1}{|\Gr|} \sum_{\rho} d_\rho \langle \FF{\G} (\rho), \FF{\G} (\rho) \rangle \| \FF{\G} (\rho)\|_{o} 
	\le
	\frac{|\G|}{|\Gr|} \sum_{\rho} d_\rho \langle \FF{\G} (\rho), \FF{\G} (\rho) \rangle = |\G|^2 \,,
	\]
	because, clearly, $\| \FF{\G} (\rho)\|_{o} \le |\G|$.  
	It means that for any representation $\rho$ either $\| \FF{\G} (\rho)\| = 0$ (and hence $\| \FF{\G} (\rho)\|_o = 0$) or $\| \FF{\G} (\rho)\| \ge \| \FF{\G} (\rho)\|_{o} = |\G|$ (alternatively, one can use the usual calculations, namely, $\sum_{\gamma \in \G} \rho(\gamma \gamma_*) = \sum_{\gamma \in \G} \rho(\gamma) \cdot \rho (\gamma_*)$ for any $\gamma_* \in \G$  but then one needs to be careful with divisors of zero). 
	Another application of \eqref{f:Parseval_representations} gives us
	\begin{equation}\label{tmp:01.10_1}
	|\G| = \frac{1}{|\Gr|} \sum_{\rho} d_\rho \|\FF{\G} (\rho) \|^2  \ge |\G| \cdot \frac{1}{|\Gr|} \sum_{\rho} d_\rho \|\FF{\G} (\rho) \| = |\G| \| \G \|_W \,.
	\end{equation}
	Hence $\| \G \|_W \le 1$
	as required.
$\hfill\Box$
\end{proof}

\bigskip 

Lemma \ref{l:G_Wiener} implies a result on  growth in the affine group relatively to some subgroups.
Namely, 
the following 
Corollary \ref{c:Aff_U} can be considered as a "baby"\,--version of our 
main results on intersections of $A^n$ with parabolic subgroups. 
Clearly, the standard Borel subgroup $B = (\lambda u| 0 \lambda^{-1})$ of the upper--triangular matrices isomorphic to a subset of $\Aff(\F_q)$ via the map $\_phi ((\lambda u| 0 \lambda^{-1})) = (\lambda^2\,  \lambda u | 0 1)$ with $\mathrm{Ker}\, \_phi = \pm I$.
The representation theory of $B$ is similar to the representation theory of $\Aff(\F_q)$ (there are $q+3$ conjugation classes, further, there exists $q-1$ one--dimensional representations and four representations of dimension $(q-1)/2$).
Hence we can apply Corollary \ref{c:Aff_U} in our studying of growth in $\SL_2 (\F_q)$.

\begin{corollary}
	Let $A\subseteq \Aff (\F_q)$ be a set,  and let 
	$\G \subseteq \Aff (\F_q)$ be a subgroup such that for any non--trivial multiplicative character $\chi$ there is $\gamma = (a,b) \in \Gamma$ such that $\chi(a) \neq 1$. 
	Also, let $z\in \Aff (\F_q)$ be an arbitrary element, $n\ge 1$ be a positive integer and  $|A|^n |\G|^2 > q^{n+2} (q-1)^2$.
	Then  $A^n \cap z\G \neq \emptyset$ and $A^n \cap \G z \neq \emptyset$.   
\label{c:Aff_U}
\end{corollary}
\begin{proof} 
	The representation theory of $\Aff (\F_q)$ is well--known see, e.g., \cite{Celniker}.
	Namely, there are $(q-1)$ one--dimensional representations $\rho_{\chi}$, which are given by multiplicative characters $\chi$, where $\rho_\chi((ab|01)) := \chi(a)$ 
	and a certain  $(q-1)$--dimensional representation $\pi$. 
	Using formula \eqref{f:Parseval_representations} with $f=A$, we have 
	\begin{equation}\label{f:Fourier_est_Aff}
	\| \FF{A} (\pi) \|_{o} < \left(\frac{|A| |\Aff(\F_q)|}{q-1} \right)^{1/2} = (|A| q)^{1/2} \,.
	\end{equation}
	Further by the assumption for any  non--trivial multiplicative character $\chi$ there is $\gamma = (a,b) \in \Gamma$ such that $\chi(a) \neq 1$. 
	It means that for any 
	such 
	$\chi$ one has $\rho_{\chi} (\G) = 0$.  
	Applying bound \eqref{f:Fourier_est_Aff}, Lemma \ref{l:G_Wiener}  and using formula \eqref{f:Parseval_representations} again, we obtain 
\[
	|A^n \cap z\G| = \frac{|A|^n |\G|}{|\Aff (\F_q)|} + \frac{q-1}{|\Aff (\F_q)|} \langle \FF{A}^n (\pi), \FF{\G} (\pi) \rangle 
	\ge 
	\frac{|A|^n |\G|}{|\Aff (\F_q)|} - (|A| q)^{n/2} > 0 \,, 
\]
	provided $|A|^n |\G|^2 > q^{n+2} (q-1)^2$.  	
	This completes the proof.
$\hfill\Box$
\end{proof}

\bigskip 

The condition $|A|^n |\G|^2 > q^{n+2} (q-1)^2$ effectively works if, roughly, $|A| \gg q^{1+\eps}$, where $\eps>0$ is a certain number. 
Further, an example of subgroup $\G$ from   Corollary \ref{c:Aff_U} is a torus $(\la 0|0\la^{-1})$, where $\la$ runs over $\F_q^*$. 
In contrary, 
if, say, $\G$ is 
the unipotent subgroup $U \subseteq \Aff (\F_q)$, then  one can easily construct a set $A$, $|A| \gg q^2/n$ such that $A^n \cap U  = \emptyset$.

\section{Some facts about Chevalley groups}
\label{sec:Chevalley}

We  recall quickly some properties of Chevalley groups. 
The detailed description of such groups can be found in many books and papers, see, e.g, classical book \cite{Steinberg} and paper \cite{Carter}. 

Let $p$ be a prime number, $q=p^s$ and  $\F_q$ be the finite field of size $q$. 
Also, let $\Phi$ be a root system, $\Pi$ its fundamental subsystem, $\Pi \subseteq \Phi^{+}$,  
$\Phi = \Phi^{+} \bigsqcup (-\Phi^{+})$.
Everything below depends on the root system $\Phi$ (and hence on $\Pi$, $\Phi^{+}$,  $-\Phi^{+}$ and so on) but we do not emphasis on this. 
Let $B$ be a Borel subgroup of $\Gr = \Gr(q)$, $U = O_p (B)$, $B=UH$ (the product is direct and $U$ is normal in $B$), $N=N(H)$ with $H$ an abelian $p'$--group (Cartan subgroup). 
The unipotent subgroup $U$ is the direct product of subgroups $\prod_{r\in \Phi^{+}} U_r$ and each $U_r$ isomorphic to the field $\F_q$. 
The Weyl group $W=N / H$ is a group generated by fundamental reflections $w_{r_1}, \ldots, w_{r_l}$, $l=|\Pi|$ and $W$ acts on the root system $\Phi$. 
When there is no problem with coset representatives we will consider $s\in W$ as an element of $\Gr(q)$. 
For $w\in W$ let $l(w)$ be the {\it length} of $w$, that is, the minimal $n$ such that $w=w_{r_1} \dots w_{r_n}$ with  $r_j \in \Pi$. 
Another description of $l(w)$ is $l(w) = |\Phi^{+} \cap w^{-1} (-\Phi^{+})|$ and it is known that $l(w) = 0$ iff $w=1$ 
(and iff $w(\Pi) = \Pi$ and iff $w(\Phi^{+}) = \Phi^{+}$).
For any $\emptyset \neq J \subseteq \Pi$ let $W_J$ be a subgroup of $W$ generated by $w_r$, where $r\in J$.  
It is well--known that for any Chevalley group the Bruhat decomposition takes place, namely, 
\begin{equation}\label{f:Bruhat}	
	\Gr = \bigsqcup_{w\in W} B w B \,,
\end{equation}
where the union in \eqref{f:Bruhat} is disjoint. 
It follows from the fact that for any fundamental root $r$ and an arbitrary $w\in W$ one has
\begin{equation}\label{f:B_inclusions}
	w_r B w \subseteq BwB \cup B w_r w B \,. 
\end{equation}
Decomposition \eqref{f:Bruhat} 
can be refined further. 
For $w\in W$ put 
$$
	U'_w = \langle \{ U_r ~:~ r\in \Phi^{+}\,, w(r) \in \Phi^{+} \} \rangle 
	\quad 	\mbox{ and } \quad 
	U''_w = \langle  \{ U_r  ~:~ r\in \Phi^{+}\,, w(r) \in -\Phi^{+} \} \rangle \,.
$$
Then, clearly,  $U= U'_w U''_w$, $B=H U'_w U''_w$ and $wU'_w w^{-1} \subseteq U$.  
Thus \eqref{f:Bruhat} can be 
transformed 
as 
\begin{equation}\label{f:Bruhat'}	
	\Gr = \bigsqcup_{w\in W} B w U''_w \,,
\end{equation}
and any element of $\Gr$ can be written in form \eqref{f:Bruhat'} uniquely. 
In particular, 
\begin{equation}\label{f:size_G}
	|\Gr| = |B| \sum_{w\in W} |U''_w| = |H| |U|  \sum_{w\in W} |U''_w| = (q-1)^{|\Pi|} q^{|\Phi^{+}|}  \sum_{w\in W} q^{l(w)} \,.
\end{equation}
From the Bruhat decomposition and the properties of Chevalley groups, it follows that all subgroups containing $B$ are $2^{l}$ subgroups of the form $P_J:= B W_J B$
and they are called {\it parabolic subgroups}. 
It is known that $N(P_J) = P_J$, and 
$$P_J = \langle B, \{w_j\}_{j\in J} \rangle = \langle B, \prod_{j=1}^J w_j \rangle = \langle B, (\prod_{j=1}^J w_j) B (\prod_{j=1}^J w_j)^{-1} \rangle \,.$$
Put $W^J = \{ w\in W ~:~ w(r) \in \Phi^{+} \mbox{ for all } r\in J \}$. 
One can check 
that any $w\in W$ can be decomposed uniquely as $w=w^J w_J$, where $w^J \in W^J$ and $w_J \in W_J$ and, moreover, $l(w) =  l(w^J) + l(w_J)$.
Any $W_J$ (and $W$ in particular) contains the unique {\it longest} element and this element is an involution. 
Formula \eqref{f:size_G} says that, basically, the length of this longest element determines size of $P_J$. 

In paper \cite{LS_representations} it was proved that Chevalley groups are quasi--random in the sense of Gowers \cite{Gowers_quasirandom} (also, see the first paper \cite{SX} where this conception was used). 
Namely, we have by \cite{LS_representations} (a similar result takes place for any simple algebraic group $\Gr$) that
\begin{equation}\label{f:d_min}	
	d_{\min} (\Gr) \gg_d q^r \,,
\end{equation}
where rank $r$ is the dimension of its maximal tori of $\Gr$ 
and $d$ is dimension of $\Gr$.

\bigskip

Let $\Pi_1 (\Gr(q)) \ge  \Pi_2 (\Gr(q)) \ge  \dots $ be sizes of  maximal proper parabolic subgroups of $\Gr(q)$. 
Consider the quantity 
\[
	P(\Gr(q)) := \min \{t ~:~ \forall H \le \Gr,\, |H| > t \implies H \mbox{ is parabolic} \} \,.
\]
In other words, $P(\Gr(q))$ 
coincides with 
size of the largest (by cardinality) non--parabolic subgroup.
The quantity depends on the concrete  Chevalley group $\Gr(q)$ 
(e.g., ${\rm P\Omega}^{+}_8 (q)$ contains the largest (by cardinality) parabolic subgroup $P$ and also  two large non--parabolic subgroups $\Omega_7(q)$, ${\rm Sp}_6 (q)$, depending on the parity of $q$,  $|\Omega_7(q)| \sim |{\rm Sp}_6 (q)| \sim q^{-1} |\Pi_1({\rm P\Omega}^{+}_8 (q))|$ see \cite[Table 6]{AB_large}). 
Nevertheless, we give a simple upper bound for $P(\Gr(q))$.
Our proof is hugely based on book \cite{KL_subgroups_book} (which in turn uses the famous Aschbacher Theorem \cite{Aschbacher}, see a good survey \cite{King}) and follows paper \cite{AB_large}, where the authors give a list of all maximal subgroups $H$ of Chevalley groups, having large size, namely, $|H| \ge |\Gr(q)|^{1/3}$. 
It is easy to see that usually maximal parabolic  subgroups of $\Gr(q)$ are even larger (clearly, $|B| \ge (|\Gr(q)| |H|)^{1/2}$)
and hence it is enough to check all "large"\, subgroups from  \cite{AB_large}.

\begin{lemma}
	Let $q$ be a  sufficiently large  number. 
	Then  we have $P({\rm PSL}_2 (q)) \le  2(q+1)$,\\ $P({\rm PSL}_3 (q)) \le q^3$,    
	and for $n\ge 4$ the following holds $P({\rm PSL}_n (q)) \le  q^{\frac{n(n+1)}{2}}$, provided $q$ is a non--square.
	\\ 
	Further, we consider $n \ge 3$ for ${\rm SU}_n (q)$, $n\ge 4$ 
	for ${\rm PSp}_n (q)$,
	$n\ge 7$ and $q$ is odd for ${\rm \Omega^{\eps}_n} (q)$, where $\eps = \pm$. 
	In all cases above with an odd $q$ and for all simple exceptional groups 
	one has  
\begin{equation}\label{f:P_2}
	q P (\Gr(q)) \le \Pi_1 (\Gr(q)) = \max \{|H| ~:~ H\le \Gr (q),\, H \neq \Gr (q)\} \,.
\end{equation}
\label{l:P_2}
\end{lemma}
\begin{proof} 
	We use Tables 3.5A--3.5F from \cite{KL_subgroups_book} to determine sizes of maximal subgroups of $\Gr(q)$, calculations from paper \cite{AB_large}, 
	as well as the Aschbacher classification Theorem, which says that every maximal subgroup of a classical group belong to one of the geometric classes 
	$\mathcal{C}_1$--$\mathcal{C}_8$ and an additional exceptional class $\mathcal{S}$. 
	For exceptional groups we consult book \cite{Wilson}.
	Due to the existence of isomorphisms between low--dimensional classical groups (see \cite[Proposition 2.9.1]{KL_subgroups_book}, for example), we may assume without loosing of the generality that $n$ satisfies the stated lower bounds.

	Let $d=(n,q-1)$, $\a = (2,q-1)$ and let us begin with ${\rm PSL}_n (q)$.
	For small $n$ it follows from the classification of subgroups of ${\rm PSL}_2 (q)$ (see, e.g., \cite{Dickson}, we use the assumption that $q$ is a non--square to avoid the subgroup ${\rm PGL}_2 (\sqrt{q})  \subset  {\rm PSL}_2 (q)$, say), 
	further for  ${\rm PSL}_3 (q)$ (we apply the assumption that $q$ is a non--square to avoid  the subgroup ${\rm PSU}_3 (q)$, say) see \cite{Mitchell}, for ${\rm PSU}_3 (q)$, ${\rm PSp}_4 (q)$ with odd $q$, again, see \cite{Mitchell} and, finally, for ${\rm PSL}_4 (q)$ with even $q$, see \cite{Mwene} 
	(here we appeal to the fact that that ${\rm PSL}_4 (q)$ contains ${\rm PSp}_4 (q)$ having size less than $q^{4(4+1)/2}$).
	Now let $n\ge 4$ and let us do not consider subgroups of the class $\mathcal{S}$ at the beginning. 
	In this case the only subgroups belonging to Aschbacher's class $\mathcal{C}_1$ are maximal parabolic subgroups $\Pi_m$ with 
	\begin{equation}\label{f:Pi_m}
		|\Pi_m| = d^{-1} q^{m(n-m)} (q-1) |\SL_m (q)| |\SL_{n-m} (q)| \sim q^{n^2-nm+m^2-1} \ge q^{\frac{3n^2}{4}-1} > q^{\frac{n(n+1)}{2}} \,.
	\end{equation}
	For $H\in \mathcal{C}_2$, we have with $t\ge 2$ that $|H| = \frac{(q-1)^{t-1} t!}{d} |\SL_{n/t} (q)|^t \ll q^{n^2/t-1}$ and this is smaller than $q^{\frac{n(n+1)}{2}}$.
	For $H\in \mathcal{C}_3$, one has  $|H| = \frac{k}{d (q-1)} |{\rm GL}_{n/k} (q^k)|$, where $k|n$, and $k$ is a prime number.
	Thus again $|H| \le q^{n^2/k-1}$. 
	If $H\in \mathcal{C}_4$, then $|H| = d^{-1} |\SL_a (q)| |\SL_{n/a} (q)| (q-1,a,n/a)$, where $2\le a <n/2$.
	In other words, $|H| \ll q^{n^2/a^2+a^2-2} \le q^{n^2/4+2} \le q^{n(n+1)/2}$. 
	Further, for $H\in \mathcal{C}_5$, we have $|H| = (q_0-1)^{-1} (q_0-1,(q_0^k-1)d^{-1}) |\SL_n (q_0)|$ with $q=q_0^k$ and $k$ is a prime number.
	Hence $|H| \ll q_0^{n^2-2} \le q^{(n^2-2)/k} \le q^{n(n+1)/2}$. 
	If $H\in \mathcal{C}_6$, then $|H| \le r^{2m} |{\rm Sp}_{2m} (q)|$, where $n=r^m$, $r|(q-1)$ and $r$ is an odd prime number. 
	It follows that $|H| \le n^2 q^{m(2m+1)}$ and this quantity is very small. 
	For $H\in \mathcal{C}_7$, one has  $|H|< |\SL_a(q)|^t/t!$, $n=a^t$, $a\ge 3$, $t\ge 2$ and again this is very small. 
	Finally, if $H\in \mathcal{C}_8$, then either $H = {\rm PSp}_n (q)$ (and we have $|{\rm PSp}_n (q)| \le q^{\frac{n(n+1)}{2}}$)
	or $|H| = |{\rm SO}^\eps_n (q)| \le 2 \a q^{n(n-1)/2}$ or $H={\rm U}_n(q_0)$, $q=q_0^2$ and $n\ge 3$. 
	In view of \eqref{f:Pi_m} we see that $P({\rm PSL}_n (q)) \le  q^{\frac{n(n+1)}{2}}$ provided $n\ge 4$.

	To finish the proof of our result in the case of ${\rm PSL}_n (q)$ it remains to consider subgroups of the class $\mathcal{S}$.
	We have $q^{\frac{n(n+1)}{2}} \ge q^{\frac{n^2-1}{3}}$ and $|\Pi_1| = q^{n^2-n} \ge q^{\frac{n(n+1)}{2}+1}$. 
	In view of  \cite[Theorem 4, Table 6]{AB_large} for large $q$ 
	(in the case of {\it all} groups ${\rm PSL}_n (q)$, ${\rm SU}_n (q)$,  ${\rm PSp}_n (q)$, ${\rm \Omega^{\eps}_n} (q)$)
	just three subgroups survive, namely,  
	${\rm P\Omega}^{+}_8 (q)$ (it contains ${\rm \Omega}_7 (q)$, ${\rm Sp}_6 (q)$ and smaller subgroups), ${\rm \Omega}_7 (q)$, 
	and ${\rm PSp}_{6} (q)$  
	(the last two contain ${\rm G}_2(q)$, $|{\rm G}_2(q)|= q^6(q^2-1)(q^6-1) \le q^{14}$). 
	The group ${\rm G}_2(q)$ in ${\rm \Omega}_7 (q)$,  ${\rm PSp}_{6} (q)$ is too small because it is easy to see that 
	$\Pi_1 ({\rm \Omega}_{7} (q)) \sim q^{16} \sim \Pi_1 ({\rm PSp}_{6} (q))$ (or consult estimate \eqref{f:Pi_m_S}, \eqref{f:Pi_m_O}  below). 
	Similarly, for 	${\rm P\Omega}^{+}_8 (q)$ sizes of  ${\rm \Omega}_7 (q)$, ${\rm Sp}_6 (q)$ do not exceed $q^{-1} \Pi_1 ({\rm P\Omega}^{+}_8 (q))$. 
	Thus indeed $P({\rm PSL}_n (q)) \le  q^{\frac{n(n+1)}{2}}$  for $n\ge 4$
	and we have proved \eqref{f:P_2} in the case $\Gr(q) = {\rm PSL}_n (q)$.

	In the general case it is sufficient to have deal with subgroups of the classes $\mathcal{C}_1$--$\mathcal{C}_8$ and we begin with parabolic subgroups. 
	For such subgroups we have analogues of formula \eqref{f:Pi_m}, namely, (see \cite[Propositions 4.1.18--4.1.20]{KL_subgroups_book})
\begin{equation}\label{f:Pi_m_U}
	|\Pi_m ({\rm SU}_{n} (q))| \sim q^{2nm -3m^2 +2} |{\rm L}_m (q^2)| |{\rm U}_{n-2m} (q)| \sim q^{n^2-2nm+3m^2-1} \,, 
\end{equation}
\begin{equation}\label{f:Pi_m_S}
	|\Pi_m ({\rm PSp}_{n} (q))| = q^{nm + m/2-3m^2/2} (q-1) |{\rm PGL}_m (q)| |{\rm PSp}_{n-2m} (q)| \sim q^{\frac{n^2-2nm+n+3m^2-m}{2}} \,,
\end{equation}
	and for $m\le \frac{n}{2}$ (we do not consider smaller parabolic subgroups)  one has 
\begin{equation}\label{f:Pi_m_O}
	|\Pi_m ({\rm \Omega^{\eps}_n} (q))| \sim q^{nm - m/2-3m^2/2} |{\rm GL}_m (q)| |{\rm \Omega^\eps_{n-2m}} (q)| \sim q^{\frac{n^2-2nm-n+3m^2+m}{2}} \,.
\end{equation}
	Analysing Tables 3.5A--3.5F from \cite{KL_subgroups_book} and using \cite[Propositions 4.1.3, 4.1.4, 4.1.6]{KL_subgroups_book}, one can easily see that  others subgroups of the class $\mathcal{C}_1$ are smaller for these parabolic groups and have form 
	${\rm GU}_{m} (q) \perp {\rm GU}_{n-m} (q)$, ${\rm PSp}_{m} (q) \perp {\rm PSp}_{n-m} (q)$, ${\rm O^\eps_{m}} (q) \perp {\rm O^\eps_{n-m}} (q)$,
	correspondingly,  with some additional restrictions on $n,m,\eps$ (e.g., $2\le m < n/2$, $m$ is even for  ${\rm PSp}_{m} (q) \perp {\rm PSp}_{n-m} (q)$. 
	In the case of the orthogonal group we use the assumption that $q$ is odd. 
	More precisely, using \eqref{f:Pi_m_U}---\eqref{f:Pi_m_O}, we check that \eqref{f:P_2} holds at least for all subgroups of the class $\mathcal{C}_1$.

	After that we apply the results from \cite{AB_large} (notice that $q^{-1} \Pi_1 (\Gr(q)) \ge |\Gr(q)|^{1/3}$ and thus it requires to use the list of the subgroups from this paper) to show that almost all other subgroups of the classes $\mathcal{C}_2$--$\mathcal{C}_8$ are obviously small.
	In the case of  ${\rm SU}_{n} (q)$ it remains to check ${\rm Sp}_{n} (q) \in \mathcal{C}_5$ with  $|{\rm Sp}_{n} (q)|\le q^{n(n+1)/2} \le q^{-1} \Pi_1 ({\rm SU}_{n} (q))$.
	If $\Gr(q) = {\rm PSp}_{n} (q)$, then all subgroups are smaller than $q^{-1} \Pi_1 (\Gr(q))$.
	Finally, in the case of ${\rm \Omega^{\eps}_n} (q)$ it remains check 
	$\mathcal{C}_2$--subgroup $H$ of size $t! |{\rm \Omega^{\eps'}_{n/t}} (q)|^t \ll q^{n(n-t)/2t}$, $t|n$, $t\ge 2$ and $H = {\rm GL}_{n/2} (q)$  
	and both of these subgroups  are less than $q^{-1} \Pi_1 (\Gr(q))$ because $n\ge 6$.
%
%
	The class $\mathcal{C}_8$ exists only for ${\rm PSp}_{n} (q)$ and it coincides with the only subgroup ${\rm O^{\pm}_{n}} (q)$, $q$ is even
	with $|{\rm O^{\pm}_{n}} (q)| \le q^{-1} \Pi_1 ({\rm PSp}_{n} (q))$. 

	It remains to consider the exceptional groups. 
	In this case we use \cite[Theorem 5, Table 2]{AB_large}, which says that any maximal subgroup $H$ of $\Gr(q)$ of size $|H| \ge |\Gr(q)|^{1/3}$ is either a maximal parabolic subgroup or belongs to a certain list, see  \cite[Table 2]{AB_large} 
	(again it is easy to check that the condition $q^{-1} \Pi_1 (\Gr(q)) \ge |\Gr(q)|^{1/3}$ takes place). 
	Analysing this Table, one can see that sizes of all non--parabolic subgroups of the exceptional groups do not exceed $|B|$ with four exceptions:
	${\rm F}_{4} (q)$ (the largest subgroups are ${\rm B}_{4} (q)$, ${\rm C}_{4} (q)$), further, ${\rm E}^\eps_{6} (q)$ (the largest subgroup is ${\rm F}_{4} (q)$),
	${\rm E}_{7} (q)$ (the largest subgroup is $(q-\eps) {\rm E}^\eps_{6} (q)$) 
	and, finally, ${\rm E}_{8} (q)$ with the largest  subgroup  ${\rm A}_{1} (q) {\rm E}_{7} (q)$.
	For ${\rm F}_{4} (q)$ consult \cite[Section 4.5.9]{Wilson} to see that there is a parabolic subgroup $H\le {\rm F}_{4} (q)$ such that 
	$$|H| = q^{15} (q-1) |{\rm Sp}_{6} (q)| \sim q^{37} \sim q |{\rm B}_{4} (q)| \sim q |{\rm C}_{4} (q)| \,,$$
	further, for ${\rm E}^\eps_{6} (q)$ see \cite[Section 4.6.4]{Wilson} and \cite{KW_E6}, where it was proved that there exists  a parabolic subgroup of size 
	$q^{25} (q-1) |{\rm L}_{2} (q)| |{\rm L}_{5} (q)| \sim q^{53} \sim q |{\rm F}_{4} (q)|$.
	Finally,  if we consider ${\rm E}_{8} (q)$, then by \cite[Section 4.7.2]{Wilson} this group contains a subgroup of size $\gg q^{58} |{\rm E}_{7} (q)|$ and this is much larger than $q|{\rm A}_{1} (q)| |{\rm E}_{7} (q)|$, if we take ${\rm E}_{7} (q)$, then, similarly, by \cite[Section 4.7.3]{Wilson} we see that   $q^2 |{\rm E}^\eps_{6} (q)|$ is small. 
	One can use another way to prove that the maximal (by size) maximal parabolic subgroup is large: just analyse the Dynkin diagrams for ${\rm F}_{4} (q)$, ${\rm E}^\eps_{6} (q)$, ${\rm E}_{7} (q)$ and ${\rm E}_{8} (q)$. 
This completes the proof.
$\hfill\Box$
\end{proof}


\bigskip 

We need a simple general lemma (a similar result can be found in \cite{Hamidoune}). 

\begin{lemma}
	Let $\Gr$ be a group and $\G_1, \G_2 \le \Gr$.
	Then 
\[
	\max_{x,y\in \Gr} |x\G_1 \cap \G_2 y| = \max_{x\in \Gr}  |x\G_1 \cap \G_2 x| \,,
\]
	and $|\G_1 \cap \G_2| \ge |\G_1| |\G_2| /|\Gr|$. 
\label{f:r_xy}
\end{lemma}
\begin{proof} 
	If the intersection $x\G_1 \cap \G_2 y$ is empty, then there is nothing to prove. 
	Otherwise for any $c\in x\G_1 \cap \G_2 y$  one has 
\[
	x\G_1 \cap \G_2 y = ((x \G_1 x^{-1}) \cap \G_2) c = c ((y^{-1} \G_2 y) \cap \G_1)
\]
	as required.

	Now from the Dirichlet principle there is $x\in \Gr$ such that $A:=x\G_1 \cap \G_2$ has size at least $|\G_1| |\G_2|/|\Gr|$.
	But $A\subseteq \G_2$ and hence $A^{-1} A \subseteq \G_1 \cap \G_2$. 
	It remains to notice that  $|A^{-1} A| \ge |A| \ge |\G_1| |\G_2|/|\Gr|$. 
	An alternative way of the proof is just use the formula  $|\G_1 \cap \G_2| = |\G_1| |\G_2|/ |\G_1 \G_2| \ge |\G_1| |\G_2|/|\Gr|$.
		This completes the proof.
	$\hfill\Box$
\end{proof}

\bigskip 

Now we are ready to prove a result on 
an 
upper bound for $|P\cap P^g|$ for parabolic subgroups $P$ of $\Gr(q)$. 

\begin{lemma}
	Let $\Gr (q)$ be a Chevalley group and $P \subset \Gr (q)$ be a parabolic subgroup. 
	Then for any $g\notin P$ one has 
\begin{equation}\label{f:P^g_cap_P}
	r_{PgP} (x) \le \frac{2|P|}{q} \quad \quad \mbox{ for all } \quad \quad x\in \Gr(q) \,. 
\end{equation}
\label{l:P^g_cap_P} 
\end{lemma}
\begin{proof} 
	In view of Lemma \ref{f:r_xy} it is enough to estimate $|P \cap P^g|$. 
	Let $P=P_J$. 
	From Bruhat decomposition \eqref{f:Bruhat} we can assume that $g\in W$ and moreover in view of \eqref{f:B_inclusions} and Lemma \ref{f:r_xy} we can assume that $g\in W^J$, $g\notin W_J$.  
	
	First of all, let us obtain \eqref{f:P^g_cap_P}
	for the Borel subgroup $B$ (in this case $g$ just from $W$).  
	The equation $Bg=gB$ can be rewritten as $Bg = HU'_g g U''_g$ and hence by \eqref{f:Bruhat'} it has $|H U'_g| = |B|/ |U''_g| = |B| q^{-l(g)}$ solutions. 
	Clearly, $l(g) \ge 1$ and the result follows (in this case we do not even need  the constant two in inequality \eqref{f:P^g_cap_P} and this is absolutely sharp, take, e.g.,  $\Gr(q)=\SL_2 (\F_q)$). 
	It is easy to see that inequality \eqref{f:P^g_cap_P} is, actually,  equality in this case.  

	Now let $P$ be an arbitrary parabolic subgroup. 
	Using the Bruhat decomposition and the arguments as in the case of the Borel subgroup, we obtain 
\[
	|P \cap P^g| = \sum_{v_1,v_2 \in W_J} |gBv_1 U''_{v_1} \cap B v_2 U''_{v_2} g| 
		\le 
			2 \sum_{v_1,v_2 \in W_J,\, l(v_2) \le l(v_1)} |gBv_1 U''_{v_1} \cap B v_2 U''_{v_2} g|
			=
\]
\begin{equation}\label{tmp:10.12_1} 
			=
			2 |B|^{-2} \sum_{v_1,v_2 \in W_J,\, l(v_2) \le l(v_1)} q^{l(v_1) + l(v_2)} |gBv_1 B \cap B v_2 B g| \,.
\end{equation}
	Now using \eqref{f:B_inclusions}, we see that for any $v\in W_J$ one has  $gBv \subseteq BvB \cup BgvB$.
	Since $g \notin P$, we get $gBv \subseteq BgvB$ and hence any element $gbv_j$, $b\in B$, $j=1,2$ can be written as $b_1 gv_j b_2$, $b_1,b_2 \in B$.
	The same is true for $vBg$, of course.  
	Whence recalling \eqref{tmp:10.12_1}, we get 
\[
	|P \cap P^g| \le 2 \sum_{v_1,v_2 \in W_J,\, l(v_2) \le l(v_1)} q^{l(v_1) + l(v_2)} |gv_1 B \cap B v_2 g| \,.
\]
	Again, applying the Bruhat decomposition and transforming $gv_1 B$ as $H U'_{gv_1} gv_1 U''_{gv_1}$, we derive
\[
	|P \cap P^g| \le 2 |B| \sum_{v_1,v_2 \in W_J,\, l(v_2) \le l(v_1),\, gv_1 = v_2 g} q^{l(v_1) + l(v_2) - l(gv_1)}
		\le
\]
\[
		\le
			2 |B| \sum_{v_1,v_2 \in W_J,\, l(v_2) \le l(v_1),\, gv_1 = v_2 g} q^{2l(v_1) - l(gv_1)} \,.
\]
	But  $g\in W^J$ and hence $l(gv_1) = l(g) + l(v_1)$.
	Clearly, $l(g) \ge 1$ because otherwise $g\in W_J$. 
	In view of \eqref{f:size_G} it gives us 
\[
	|P \cap P^g| \le 2 |B| q^{-1} \sum_{v_1,v_2 \in W_J,\, l(v_2) \le l(v_1),\, gv_1 = v_2 g} q^{l(v_1)} 
	\le
	2|B| q^{-1} \sum_{v \in W_J}  q^{l(v)} = 2|P| q^{-1} 
\]	
	as required.
$\hfill\Box$
\end{proof}

\bigskip

It is easy to see that estimate \eqref{f:P^g_cap_P} is tight up to constants  (consider parabolic subgroups of  $\SL_n(\F_q)$, say).

We finish this Section by a lemma in the spirit of the well--known result of Frobenius \cite{Frobenius} on the representation of $\SL_n(\F_q)$.


\begin{lemma}
	Let $\Gr (q)$ be a Chevalley group and $P \subseteq \Gr (q)$ be a parabolic subgroup. 
	Suppose that $\rho$ is an arbitrary non--trivial irreducible representation of $P$ such that $\FF{H} (\rho) \neq 0$.
	Then $d_\rho \ge \frac{q-1}{2}$. 
\label{l:Frobenius_Chevalley}
\end{lemma}
\begin{proof} 
	At the beginning let $P$ be a Borel subgroup $B$
	and suppose that 
	$\rho(h) = 1$ for any $h\in H$.	
	We know that $B=UH$ and thus there is $r\in \Phi^{+}$ such that $\rho(U_r) \neq 1$. 
	Since there is a canonical homomorphism from $\SL_2(\F_q)$ onto $\langle U_r, U_{-r}\rangle$, where $r\in \Phi$ is an arbitrary and 
\[
\left( {\begin{array}{cc}
		\la  & 0 \\
		0 & \la^{-1} \\
\end{array} } \right) 
\left( {\begin{array}{cc}
		1 & t \\
		0 & 1 \\
\end{array} } \right) 
\left( {\begin{array}{cc}
	\la^{-1} & 0 \\
	0 & \la \\
	\end{array} } \right) 
=\left( {\begin{array}{cc}
	1 & \la^{2} t \\
	0 & 1 \\
	\end{array} } \right) 
\]	
	we see that, say, $g:= (11|01)$ is conjugated with $g^m$, where $m$ runs over all quadratic residues of $\F^*_q$. 
	In other words, the operation $x\to x^m$ permutes all eigenvalues of $\rho(g)$ and hence the dimension $d_\rho$ is at least $\frac{q-1}{2}$
	(strictly speaking, the arguments above hold for $\F_p$ but it is easy to show that for $\F_q$ a similar method works, see, e.g., \cite[Proposition 8.10]{Breuillard_lectures}). 

	Now we can assume that $\rho(u)=1$ for any $u\in U$ (because otherwise we can apply the arguments above)
	but there is $h_* \in H$ such that $\rho(h_*) \neq 1$.
	As we know $H$ is an abelian group equals  the product of $l=|\Pi|$ cyclic subgroups which are isomorphic to $\F^*_q$.   
	Clearly, for any $h\in H$ one has  $\FF{H} (\rho) = \rho(h) \FF{H} (\rho) = \FF{H} (\rho) \rho(h)$ in particular, 
	$\FF{H} (\rho) = \rho(h_*) \FF{H} (\rho) =  \FF{H} (\rho)  \rho(h_*)$.  
	Thus the matrix $\rho(h_*)$ has a non--trivial invariant subspace $\mathcal{L}$, corresponding to the eigenfunctions with the eigenvalue equal one because
	otherwise $\FF{H} (\rho) = 0$. 
	Obviously, matrices $\rho(h)$, $h\in H$ commute with $\rho(h_*)$ and hence $\rho(h) \mathcal{L} \subseteq \mathcal{L}$. 
	Since $B=HU$, it follows that $\rho(b)\mathcal{L} \subseteq \mathcal{L}$ for any $b\in B$. 
	But then we find an invariant  subspace of $\rho$, contradicting our assumption.


	Finally, 
	let $P=P_J$ be an arbitrary parabolic subgroup.
	Then we have $\rho$ is $1$ on $B$ because otherwise we can use the previous arguments.
	We know that
	$P_J = \langle B,  \{w_j\}_{j\in J} \rangle = \langle B, \prod_{j=1}^J w_j \rangle = \langle B, (\prod_{j=1}^J w_j) B (\prod_{j=1}^J w_j)^{-1} \rangle$ 
	and hence $\rho$ is $1$ on  $P$.
	This completes the proof.
$\hfill\Box$
\end{proof}

\section{Growth relatively to parabolic subgroups}
\label{sec:proof}

Now let us obtain a result on growth of subsets from  $\Gr(q)$ under left/right multiplications by  parabolic subgroups. 

For any sets $A,B,C$ put $\sigma_{A} (B,C) := \sum_{x \in A} r_{BC} (x)$. 
Bounds in  Theorem \ref{t:growth_in_P} below depend on the quantities $\sigma_{P} (A^{-1}, A)$, $\sigma_{P} (A, A^{-1})$, where $A$ is an arbitrary subset of $\Gr(q)$ and $P$ is a parabolic subgroup. The sense of these expressions is rather obvious, namely,   $\sigma_{P} (A^{-1}, A)$ and $\sigma_{P} (A, A^{-1})$ are small if the intersection of $A$ with left/right cosets of $P$ is small in average.

\begin{theorem}
	Let $\Gr (q)$ be a Chevalley group and $P \subset \Gr (q)$ be a parabolic subgroup. 
	Then for any set $A\subseteq \Gr (q)$ one has either 
	$$|AP| |A\cap P| \ge 2^{-1} |A|^2$$ 
	or 
\begin{equation}\label{f:growth_in_P-}
	|AP| |PA| \ge 2^{-2} |A| |P| q \,.
\end{equation}
	In particular, 
\begin{equation}\label{f:growth_in_P}
	\max\{ |AP|, |PA| \} \ge 2^{-1} \min\{ |A|^2 |A\cap P|^{-1}, (|A| |P|q)^{1/2} \} \,.
\end{equation} 
	Similarly, 
\begin{equation}\label{f:growth_in_P+}
	|APA| \ge |P|/4 \cdot  \min \{ q, |A|^4 \sigma^{-1}_{P} (A^{-1}, A) \sigma^{-1}_{P} (A, A^{-1})  \} \,,
\end{equation} 
	and if $A\subsetneq P$, then 
\begin{equation}\label{f:growth_in_P+'}
	|PAB| \ge  q |P| \,.
\end{equation} 
\label{t:growth_in_P}
\end{theorem} 
\begin{proof} 
	Let $g\notin P$ and put $A_g = A \cap gP$. 
	Also, let $\D = \max_{g\notin P} |A_g|$. 
	We have 
\[
	\E(A^{-1}, P) = \sum_{x} r^2_{A^{} P} (x) = \sum_{x \in P} r^2_{A^{} P} (x) + \sum_{x \notin P} r^2_{A^{} P} (x) 
		\le |P| \sum_{x \in P} r_{A^{} P} (x) + \D |P| |A|
		= 
\]
	\begin{equation}\label{tmp:04.10_1-}
		=
		|P|^2 |A\cap P| + \D |P| |A|
	\,. 
	\end{equation}
	In view of \eqref{f:energy_CS}, we get
	\begin{equation}\label{tmp:04.10_1-'}
		|AP| \ge 2^{-1} \min \{|A| |P| \D^{-1}, 
		|A|^2 |A\cap P|^{-1}
		\} \,.
	\end{equation}
	On the other hand, using Lemma \ref{l:P^g_cap_P}, we 
	derive 
	\begin{equation}\label{f:B_Ag-}
	\E(P,A_g) = \sum_{x} r^2_{P A_g} (x) \le \sum_{x} r_{P A_g} (x) r_{P g P} (x) \le 2|P|^2 |A_g| q^{-1} \,,
	\end{equation}
	and hence by the Cauchy--Schwarz inequality, we get 
	\begin{equation}\label{f:B_Ag}
	|PA| \ge |P A_g| \ge \frac{|P|^2 |A_g|^2}{\E(P,A_g)} \ge 2^{-1} q|A_g| = 2^{-1} q \D \,,
	\end{equation}	
	where we choose $g$ such that $|A_g| = \D$. 
	Combining \eqref{tmp:04.10_1-'} and \eqref{f:B_Ag}, we arrive to \eqref{f:growth_in_P}.

	Similarly, let us obtain \eqref{f:growth_in_P+}.
	In view of Lemma \ref{f:r_xy} and Lemma \ref{l:P^g_cap_P}, we have 
\[
	\sigma:= \sum_x r^2_{APA} (x) = \sum_{z,z'} r_{A^{-1}A} (z) r_{AA^{-1}} (z') |zP \cap P z'|
	=
\]
\[
	=
		\sum_{z,z' \in P} r_{A^{-1}A} (z) r_{AA^{-1}} (z') |zP \cap P z'| + \sum_{z,z' \notin P} r_{A^{-1}A} (z) r_{AA^{-1}} (z') |zP \cap P z'|
	\le 
\]
\begin{equation}\label{tmp:11.12_1}
	\le 
	|P| \sigma_{P} (A^{-1}, A) \sigma_{P} (A, A^{-1}) + 2|P| q^{-1} |A|^4 \,.
\end{equation}
	By the Cauchy--Schwarz inequality, we know that $\sigma |APA| \ge |A|^4 |P|^2$ and combining this with \eqref{tmp:11.12_1}, we obtain the required result.

	It remains to obtain \eqref{f:growth_in_P+'}. 
	Since $A$ does not belong to $P=P_J$,  it follows that there are $w_J \in W_J$, $1\neq w^J \in W^J$, $b_1,b_2 \in B$ such that  the product $b_1 w_J w^J b_2$ is an element from $A$. 
	It easily follows from the Bruhat decomposition. 
	Then $Pw_J w^J B \subseteq  PAB$ and in view of \eqref{f:B_inclusions}, we have $Pw_J =P$.   
	Thus we see that $PAB$ contains disjoint sets $Bvw^JB$ for any $v\in W_J$ and hence by \eqref{f:size_G}
\[
	|PAB| \ge \sum_{v \in W_J} |Bvw^JB| = |B| \sum_{v \in W_J} q^{l(v w^J)} = |B| q^{l(w^J)} \sum_{v \in W_J} q^{l(v)} \ge q |B| \sum_{v \in W_J}  q^{l(v)} = q |P_J| \,. 
\] 
	This completes the proof.
	$\hfill\Box$
\end{proof}

\begin{remark}
	It is easy to see that bound \eqref{f:growth_in_P} is tight. 
	Indeed, let $P=B$ be a Borel subgroup and  $A=B \bigsqcup Bw_r B$, where $w_r$ is a fundamental reflection.
	In particular, $l(w_r) = 1$ and $A$ is a parabolic subgroup.
	Then $AB=BA=A$ but by \eqref{f:size_G}, we have $|A| \sim  q |B| \sim \sqrt{|A||B|q}$. 
\end{remark}


Now we are ready to obtain a result on 
intersections of powers of $A$ with parabolic subgroups.  
We use quasi--random technique from \cite{Gowers_quasirandom}, \cite{SX}.

\begin{theorem}
	Let $\Gr (q)$ be a Chevalley group and $P \subset \Gr (q)$ be a parabolic subgroup. 
	Also, let $n\ge 1$ be a positive integer and  $X, Y_1,\dots, Y_n \subseteq \Gr (q)$ be nonempty sets such that $X \cap P = \emptyset$ and 
\begin{equation}\label{f:A^n_cap_B}	
	q |X| |P|^3 d^{n+2}_{\min} \cdot \prod_{j=1}^n |Y_j| \ge 4 |\Gr|^{n+4} \,.
\end{equation}
	Then $XY_1 \dots Y_n X \cap P \neq \emptyset$. 
\label{t:A^n_cap_B} 
\end{theorem}
\begin{proof}
	First of all, let us obtain a general upper bound for  $\| A (\rho) \|_{o}$, where $A$ is any subset of $\Gr=\Gr(q)$ and 
	$\rho$ is an arbitrary non--trivial representation  of $\Gr$. 
	Using formula \eqref{f:Parseval_representations} with $f=A$, we have 
	\begin{equation}\label{f:Fourier_est}
	\| \FF{A} (\rho) \|_{o} < \left(\frac{|A| |\Gr|}{d_{\min}} \right)^{1/2}  \,.
	\end{equation}
	Now if $XY_1 \dots Y_n X \cap P = \emptyset$, then $(PX)Y_1 \dots Y_n (XP) \cap P = \emptyset$. 
	In terms of the representation theory it can be rewritten as 
\[
	0 = \frac{|PX| |Y_1| \dots |Y_n| |XP| |P|}{|\Gr|} + \frac{1}{|\Gr|} \sum_{\rho \in \FF{\Gr},\, \rho \neq 1} 
		\langle \FF{PX} (\rho) \FF{Y}_1 (\rho) \dots \FF{Y}_n (\rho) \FF{XP} (\rho), \FF{P} (\rho) \rangle 
\]
	Since $X \cap P = \emptyset$, we know by estimate \eqref{f:growth_in_P-} of Theorem \ref{t:growth_in_P} that $|PX| |XP| \ge 2^{-2} |X| |P| q$. 
	Using this fact and applying Lemma \ref{l:G_Wiener}, combining with bound \eqref{f:Fourier_est} for the sets $Y_j$, we obtain 
\[
	\frac{|PX| |Y_1| \dots |Y_n| |XP||P|}{|\Gr|} 
		<
			\| P\|_W \left(\frac{|PX| |\Gr|}{d_{\min}} \right)^{1/2} \left(\frac{|XP| |\Gr|}{d_{\min}} \right)^{1/2} 
			\prod_{j=1}^n \left(\frac{|Y_j| |\Gr|}{d_{\min}} \right)^{1/2} 
			\le 
\]
\[
			\le 
			\left(\frac{|\Gr|}{d_{\min}}\right)^{(n+2)/2} \left( |PX| |XP| \prod_{j=1}^n |Y_j| \right)^{1/2}
\]
	or, in other words, 
\[
	q |X| |P|^3 d^{n+2}_{\min} \cdot \prod_{j=1}^n |Y_j| < 4 |\Gr|^{n+4} \,.
\]
This completes the proof.
$\hfill\Box$
\end{proof}

\bigskip

Let $P$ be a parabolic subgroup of size at least $|\Gr(q)|/d_{\min}$ and let $A\cap P = \emptyset$. 
Then Theorem \ref{t:A^n_cap_B} says us that $A^{n+2} \cap P \neq \emptyset$, provided 
\begin{equation}\label{f:lim_d_min} 
	|A| \gg \frac{|\Gr(q)|}{d_{\min}} \cdot \left(\frac{d^2_{\min}}{q} \right)^{1/(n+1)} \,.
\end{equation}
In other words, if we want to generate $\Gr(q)$ by powers of $A$, then  we have a 
natural barrier  
$|A| \gg \left(\frac{|\Gr(q)|}{d_{\min}} \right)^{1+\eps}$.
Our next aim is to 
relax 
the last 
condition. 

To do this in a particular case of $\SL_2 (\F_q)$  we need a result 
on growth in 
$\SL_2 (\F_q)$, which provides us some concrete bounds for growth, see \cite[Theorem 14]{MS_Zaremba} (which in turn develops the ideas of \cite{H}, \cite{RS_SL2}).
In the general case we apply Lemma \ref{l:P_2}.

\begin{theorem}
	Let $q\ge 5$, $A\subseteq \SL_2 (\F_q)$ be a generating set, $q^{2-\epsilon} \ll |A| \le q^{\frac{72}{35}}$, $\epsilon<\frac{2}{25}$. 
	Then $|AAA| \gg |A|^{\frac{25}{24}}$. 
\label{t:K_lower_SL}
\end{theorem}


Now we are ready to prove a result, which breaks the limit from \eqref{f:lim_d_min}.
The absolute constants in $2)$, $3)$ can be easily computed but we do not specify them.

\begin{theorem}
	Let $B$ be a Borel subgroup of $\SL_2 (\F_q)$ and $A\subseteq \SL_2 (\F_q)$ be an arbitrary set.  
	Then the following holds  \\
	$1)~$ If  $|A| \ge q^{2-c}$, $c<\frac{2}{25}$, then there is 
	$n\le \lceil \frac{24(1+c)}{2-25c} \rceil$ 
	such that $A^{3n+2} \cap B \neq \emptyset$.\\
	$2)~$ If $|A| \ge q^{1+\d}$, then there is 
	$n\ll 1/\d$ 
	with 
	$A^{n} \cap B \neq \emptyset$.\\
	$3)~$ In general, let $q$ an odd number, $\Gr(q)$ be a 
	Chevalley group and $P \subset \Gr (q)$ be a parabolic subgroup.  
	Suppose that $|A| \ge  \Pi_1 (\Gr(q)) q^{-1+\d}$. 
	Then there is $n$,  $n \ll_l \d^{-1}$ such that $A^{n} \cap P \neq \emptyset$. 
\label{t:A^n_cap_B_SL2}
\end{theorem}
\begin{proof}
	We can assume that $A\cap B = \emptyset$ because otherwise there is nothing to prove. 
	Let 
	$U := \{(1u|01) : u\in \F_q\}$.
	If $A$ generates $\SL_2 (\F_q)$, then by Theorem \ref{t:K_lower_SL} either $|A| \ge q^{\frac{72}{35}}$ or 
	$|AAA|\gg |A|^{\frac{25}{24}} > q^{2+\frac{2-25c}{24}}$.
	Applying Theorem \ref{t:A^n_cap_B} with $P=B$, $X=A$, $Y_j =AAA$ and $d_{\min} = \frac{q-1}{2}$ we see that 
	$A^{3n+2} \cap B \neq \emptyset$ provided $n\ge \lceil \frac{24(1+c)}{2-25c} \rceil$. 
	If $|A| \ge q^{\frac{72}{35}}$, then Theorem \ref{t:A^n_cap_B} with $P=B$, $X=A$, $Y_j =A$ gives us even better upper bound for $n$.

	Now suppose that  $A$ does not generate $\SL_2 (\F_q)$. 
	By the well--known subgroups structure of $\SL_2 (\F_q)$ see, e.g., \cite{Dickson} we have  that $A$ is a subset of  a Borel subgroup and conjugating we can assume that $A$ is a subset of the standard Borel subgroup $B_*$ of the upper--triangular matrices.  
	Also, we have $B= g^{-1} B_* g$ for a certain $g\in  \SL_2 (\F_q)$.
	We can assume that $g \notin B_*$ because otherwise $B_* = B$ and hence $A = B_* \cap A = B \cap A \neq \emptyset$.
	One can carefully use inequalities \eqref{f:growth_in_P+}, \eqref{f:growth_in_P+'} of Theorem \ref{t:growth_in_P} and prove that  $A^{-1} B A^{-1}$ has size at least
	$|\SL_2 (\F_q)| - (1+o(1)) |B|$. 
	It is not enough for our purposes and we consider $A^n$ directly. 
	By the Bruhat decomposition 
	the element $g$ 
	can be written as $b w u$, where  $b \in B$, $u \in U$ and $w=(01|(-1)0)$.   
	Then any element of $B= g^{-1} B_* g$ 
	has the form 
\begin{equation}\label{f:ABA}
\left( {\begin{array}{cc}
	1 & -v \\
	0 & 1 \\
	\end{array} } \right) 
\left( {\begin{array}{cc}
	\la & 0 \\
	u & \la^{-1} \\
	\end{array} } \right)
\left( {\begin{array}{cc}
	1 & v \\
	0 & 1 \\
	\end{array} } \right)
=
\left( {\begin{array}{cc}
	\la-v u & v(\la-u v) -  v \la^{-1} \\
	u & u v + \la^{-1} \\
	\end{array} } \right) \,,
\end{equation}
	where the variables $\la$, $u$ run over 
	$\F^*_q$, $\F_q$, correspondingly, 
	and $v$  is a fixed element.
	Since $A^n \subseteq B_*$, it follows  that it is enough to find an element $(\la\, (v\la -  v \la^{-1}) | 0 \la^{-1} ) \in B_* \cap B$ in $A^n$.
	The intersection  $T:=B_* \cap B$ is, clearly, is subgroup of size $q-1$ and $T$ is, actually, a torus. 
	Applying 
	Corollary \ref{c:Aff_U} (here we use the representation theory for $B$ not $\Aff(\F_q)$), we obtain that $A^3 \cap T \neq \emptyset$, provided $|A| \gg q^{5/3}$.

	Now let us prove that the condition $|A| \ge q^{1+\d}$ implies that there is $n \ll 1/\d$ such that $A^{n} \cap B \neq \emptyset$.
	Again, if $A$ generates $\SL_2 (\F_q)$, then we consequently apply Theorem \ref{t:K_lower_SL} (also, see \cite[Lemma 4]{RS_SL2}) 
	and derive that $A^{3^{n+1}} = \SL_2 (\F_q)$ provided $(1+c_*)^n (1+\d) > 8/3$, where $c_* >0$ is an absolute constant.
	Hence $n = O(1)$ and in particular, $A^{3^{n+1}} \cap B \neq \emptyset$. 
	Now if $A$ does not generate $\SL_2 (\F_q)$, then by the subgroups structure of $\SL_2 (\F_q)$ we see that $A$ is a subset of a Borel subgroup $B_1$. 
	Put $T = B \cap B_1$. 
	Then, as above,  $T$ is a torus, 
	having size 
	$q-1$.
	Applying Corollary \ref{c:Aff_U} again, we obtain $A^n \cap T \neq \emptyset$ provided $|A| > q^{1+2/n}$.
	Thus the restriction $n> 2/\d$ is enough in this case.

	It remains to prove the third part of our theorem. 
	Again we can assume that $A\subseteq \G \le \Gr(q)$, $\G \neq \Gr(q)$ because otherwise we consequently apply Theorem \ref{t:growth_in_Lie} to generate the whole $\Gr(q)$. 
	By our assumption and Lemma \ref{l:P_2}, we have 
$$
	|\G| \ge |A| \ge  \Pi_1 (\Gr(q)) q^{-1+\d} > \Pi_1 (\Gr(q)) q^{-1} \ge P (\Gr(q)) 
$$
	 and hence $\G$ is a parabolic subgroup, $|\G| \le \Pi_1 (\Gr (q))$. 

	Recall that the intersection of two Borel subgroups contains a maximal torus of $\Gr(q)$. 
	Indeed, by the Bruhat decomposition  we have $B:= gB_*g^{-1} = u w B_* w^{-1} u^{-1}$, where $u\in U$, $w\in W$ and hence $u H u^{-1} \subseteq B_* \cap B$ because $w^{-1} H w = H \subseteq B_*$.   
	In particular, the subgroup $P \cap \G$ contains a torus $T$. 
	Applying the arguments from the proof of Corollary \ref{c:Aff_U} for the group $\G$, as well as Lemma \ref{l:Frobenius_Chevalley}, we see that $A^n \cap T \neq \emptyset$, if 
\[
	|A| \gg \frac{|\Pi_1 (\Gr (q))|}{q} \cdot \left(\frac{|\Pi_1 (\Gr (q))|}{|T|} \right)^{2/n} \ge \frac{|\G|}{q} \cdot \left(\frac{|\G|}{|T|} \right)^{2/n}\,.
\]
	By the assumption  $|A| \ge  \Pi_1 (\Gr(q)) q^{-1+\d}$ and hence it is enough to have $n\gg_{l} \d^{-1}$.
This completes the proof.
$\hfill\Box$
\end{proof}

\bigskip 

{\bf Example.} Let $B^{+}, B^{-}$ be the standard Borel subgroups of the upper/lower--triangular matrices from $\SL_2 (\F_p)$ and $p\equiv -1 \pmod 4$. 
Let also $A\subseteq B^{+} \setminus B^{-}$ such that all elements of matrices from  $A$ are quadratic residues. 
Then one can see that $A\cap B^{-}$ and $A^2\cap B^{-}$ are empty. 
Also, we have $|A| \gg p^2$. 
It means that in Theorem \ref{t:A^n_cap_B_SL2} we need at least three multiplications even for sets $A$ with $|A| \gg p^2$.

\section{Two applications to Zaremba's conjecture}
\label{sec:application}

Using 
inequality \eqref{f:growth_in_P-} of Theorem \ref{t:growth_in_P}, combining with Theorem \ref{t:A^n_cap_B}, and applying the method from  \cite{MS_Zaremba} one can 
decrease 
the constant $30$ in Theorem \ref{t:main_MS} to $24$. 
We go further, using the specific of our problem and obtain Theorem \ref{t:main_MS_new} from the Introduction.

Denote by $F_M(Q)$  the set of all {\it rational} numbers  $\frac{u}{v}, (u,v) = 1$ from $[0,1]$ with all partial quotients in (\ref{exe}) not exceeding $M$ and with $ v\le Q$:
\[
F_M(Q)=\left\{
\frac uv=[0;b_1,\ldots,b_s]\colon (u,v)=1, 0\leq u\leq v\leq Q,\, b_1,\ldots,b_s\leq M
\right\} \,.
\]
By $F_M$ denote the set of all {\it irrational} 
numbers  from $[0,1]$ with partial quotients less than or equal to $M$.
From \cite{hensley1992continued} we know that the Hausdorff dimension $w_M := \mathrm{HD} (F_M)$ of the set  $F_M$ satisfies
\begin{equation}
w_M = 1- \frac{6}{\pi^2}\frac{1}{M} -
\frac{72}{\pi^4}\frac{\log M}{M^2} + O\left(\frac{1}{M^2}\right),
\,\,\, M \to \infty \,,
\label{HHD}
\end{equation}
however here 
it is enough for us to have 
a simpler result from \cite{hensley1989distribution}, which states that
\begin{equation}\label{oop}
1-w_{M} \asymp  \frac{1}{M}
\end{equation}
with some absolute constants in the sign $\asymp$.
Explicit estimates for dimensions of $F_M$  for certain values of $M$ can be found in \cite{jenkinson2004density}, \cite{JP} and in other papers.
For example, see \cite{JP} 
\begin{equation}\label{f:w_2}
	w_2 = 0.5312805062772051416244686...  > \frac{1}{2}
\end{equation}
In  papers \cite{hensley1989distribution,hensley1990distribution} Hensley gives the bound
\begin{equation}
|F_M(Q)| \asymp_M Q^{2w_M} \,.
\label{QLOW}
\end{equation}
More generally (see \cite{hensley1996}), let $\mathcal{A} \subset \mathbb{N}$ be a finite set with at least two points and let $F_{\mathcal{A}}$ be the set of all irrational numbers such that  $b_j\in \mathcal{A}$ (previously,  $\mathcal{A} = \{1,\dots, M\}$). 
Then  it is known \cite{hensley1989distribution}, \cite{hensley1996} that for the correspondent discrete set $F_{\mathcal{A}} (Q)$ formula  \eqref{QLOW} takes place (the constants there depend on $\mathcal{A}$ of course). 
The Hausdorff dimension $\mathrm{HD} (F_{\mathcal{A}})$ of the set $F_{\mathcal{A}}$ it is known to exist and satisfies $0<\mathrm{HD} (F_{\mathcal{A}}) < 1$.

We associate a set of matrices from $\Gr = \SL_2 (\F_p)$ with the continued fractions. 
One has 
\begin{equation}\label{f:continuants_aj}
\left( {\begin{array}{cc}
	0 & 1 \\
	1 & b_1 \\
	\end{array} } \right) 
\dots 
\left( {\begin{array}{cc}
	0 & 1 \\
	1 & b_s \\
	\end{array} } \right) 
=
\left( {\begin{array}{cc}
	p_{s-1} & p_s \\
	q_{s-1} & q_s \\
	\end{array} } \right) \,,
\end{equation}
where $p_s/q_s =[0;b_1,\dots, b_s]$ and $p_{s-1}/q_{s-1} =[0;b_1,\dots, b_{s-1}]$.
Clearly, $p_{s-1} q_s - p_s q_{s-1} = (-1)^{s}$.
Let $Q=p-1$ and consider the set $F_M(Q)$. 
Any $u/v \in F_M(Q)$ corresponds to a matrix from \eqref{f:continuants_aj} such that  $b_j \le M$.
The set $F_M(Q)$ splits into ratios with even $s$ and with odd $s$, 
in other words 
$F_M(Q) =  F^{even}_M(Q) \bigsqcup F^{odd}_M(Q)$. 
Let $A \subseteq \SL_2 (\F_p)$ be the set of matrices of the form above with even $s$.
It is easy to see from \eqref{QLOW}, multiplying if it is needed the set $F^{odd}_M (Q)$ by $(01|1b)^{-1}$, $1\le b \le M$ that 
$|F^{even}_M(Q)| \gg_M |F_M (Q)|  \gg_M Q^{2w_M}$.
Let $B$ be the standard Borel subgroup of $\SL_2 (\F_p)$, i.e., the set of all upper--triangular matrices.  
It is easy to check that  if for a certain $n$ one has $A^n \cap B \neq \emptyset$, then $q_{s-1}$ 
equals zero modulo $p$ and hence there is $u/v \in F_M ((2p)^n)$ such that $v\equiv 0 \pmod p$.
Actually, if we find any number from $p_s,q_s,p_{s-1},q_{s-1}$ equals zero modulo $p$, then we can do the same,  see \cite{hensley_SL2} (but we do not need  this fact).

\begin{lemma}
	We have 
\begin{equation}\label{f:sigma_A}
	\sigma_{B} (A, A^{-1}) \le p|A| 
	\quad \quad \mbox{ and } \quad \quad 
	\sigma_{B} (A^{-1}, A) \le M^2 p|A| \,.
\end{equation}
	Moreover,  
\begin{equation}\label{f:sigma_A_D}
	\max_{g\in \SL_2 (\F_p)} \{|A\cap gB|, |A\cap Bg| \} \le M p \,,
\end{equation}
\begin{equation}\label{f:sigma_A_D+}
\max_{g,h\in \SL_2 (\F_p)} |A\cap gBh| 	\ll_M  |A| \cdot p^{-\frac{2w_M-1}{4}}  \,.
\end{equation}
\label{l:sigma_A}
\end{lemma}
\begin{proof}
	Let us begin with the estimation of $\sigma_{B} (A, A^{-1})$. 
	We see that the product 
\begin{equation} 
\left( {\begin{array}{cc}
		p_{s-1} & p_s \\
		q_{s-1} & q_s \\
\end{array} } \right) 
\left( {\begin{array}{cc}
	q'_{t} & -p'_t \\
	-q'_{t-1} & p'_{t-1} \\
	\end{array} } \right) 
\in B\,,
\end{equation}
iff $q'_t q_{s-1} \equiv q_s q'_{t-1} \pmod p$. 
It is well--known that $\frac{q_s}{q_{s-1}} = [b_s;b_{s-1}, \dots, b_1]$ and hence the number of pairs $(q_{s-1},q_s)$ is at most $|A|$. 
Further, fixing $q'_{t-1}$, as well as a pair $(q_{s-1},q_s)$, we find $q'_t$ uniquely modulo $p$ and hence we find $q'_t$ because $q'_t \le p-1$. 
Thus $\sigma_{B} (A, A^{-1})  \le p |A|$ because all variables do not exceed $p-1$. 
The argument showing that $\sigma_{B} (A^{-1}, A) \le M p|A|$ is even simpler because in this case we have the equation 
$p'_{t-1} q_{s-1} \equiv p_{s-1} q'_{t-1} \pmod p$ and any pair $(p_{s-1}, q_{s-1}, q'_{t-1})$ determine $p'_{t-1}$.
It remains to notice that we can reconstruct $(p_s,q_s)$ from $(p_{s-1}, q_{s-1})$ in  at most $M$ ways.    
Bound \eqref{f:sigma_A_D} can be obtained exactly in the same way.

Finally, to get \eqref{f:sigma_A_D+} we see that the inclusion 
\begin{equation}\label{f:inclusion_gen_Borel}
\left( {\begin{array}{cc}
	\alpha & \beta \\
	\gamma & \delta \\
	\end{array} } \right) 
\left( {\begin{array}{cc}
	p_{s-1} & p_s \\
	q_{s-1} & q_s \\
	\end{array} } \right) 
\left( {\begin{array}{cc}
	a & b \\
	c & d \\
	\end{array} } \right) 
\in B
\end{equation}
	gives us 
\begin{equation}\label{f:A_gBh}
a(\gamma p_{s-1} +\delta q_{s-1}) \equiv -c(\gamma p_s + \delta q_s)  \pmod p \,.
\end{equation}
	We can assume that $a,c \neq 0$ because 
	this case was considered above and 
	the same situation 
	 for $\gamma =0$.
	If $\delta= 0$, then $ap_{s-1} \equiv -c p_{s} \pmod p$ and fixing $p_s$ we find $p_{s-1}$ uniquely.   
	But $\frac{p_s}{p_{s-1}} = [b_s;b_{s-1}, \dots, b_2]$ and we determine the whole matrix, choosing $b_1$ in at most $M$ ways. 
	Thus  suppose that all coefficients in \eqref{f:A_gBh} do not vanish. 
	In view of the Bruhat decomposition 
	(i.e. one can put $d = \alpha = 0$, $\beta = b =1$, $\gamma = c = -1$)
	equation \eqref{f:A_gBh} can be rewritten as 
\begin{equation}\label{f:A_gBh'}
	a(\delta q_{s-1} - p_{s-1}) \equiv  \delta q_s -p_s \pmod p 
\end{equation}
	or, in other words, 
\begin{equation}\label{f:A_gBh''}
	\delta (q_s + \omega q_{s-1}) \equiv p_s + \omega p_{s-1} \pmod p \,,
\end{equation}	
	where $\omega = -a$. 
	Equation \eqref{f:A_gBh''} can be interpreted easily: any Borel subgroup fixes a point (the standard Borel subgroup fixes $\infty$) and hence inclusion \eqref{f:inclusion_gen_Borel} says that  our set $A$ transfers $\omega$ to $\delta$.
	In other terms, identity \eqref{f:A_gBh''} says that the tuples $(q_s,q_{s-1},p_s, p_{s-1})$ belongs to a hyperspace with the normal vector 
	$(\delta,\delta \omega,-1,-\omega)$ and hence for some other solutions of \eqref{f:A_gBh''}, we get 
\begin{eqnarray}\label{f:det_A}	
	\begin{vmatrix}
		q_s & q_{s-1} & p_s & p_{s-1} \\ 
		q'_s & q'_{s-1} & p'_s & p'_{s-1} \\ 
		q''_s & q''_{s-1} & p''_s & p''_{s-1} \\ 
		q'''_s & q'''_{s-1} & p'''_s & p'''_{s-1} 
	\end{vmatrix}
	 \equiv 0 \pmod p \,.
\end{eqnarray}
	Now consider the set $\tilde{A} \subset A$ which is constructing in an analogues way from $F_M (2^{-5} Q^{1/k})$, $k=4$ but not from $F_M (Q)$.
	Our first task is to prove 
\begin{equation}\label{f:sigma_tA_D+}
	\max_{g,h\in \SL_2 (\F_p)} |\tilde{A}\cap gBh| \le M p^{1/k} \,.
\end{equation}
	Clearly, $|\tilde{A}| \sim |A|^{1/k} \sim p^{2w_M/k}$ and hence \eqref{f:sigma_tA_D+} would give us an almost square--root saving as $M$ tends to $\infty$.
	If we solve equation \eqref{f:det_A} with elements from $\tilde{A}$, then we arrive to an equation 
\begin{equation*}
	X q_s + Y q_{s-1} + Z p_s + W p_{s-1} \equiv 0 \pmod p \,,
\end{equation*}
	where $|X|, |Y|, |Z|, |W| < 2^{-2} p^{3/k}$ which is, actually, an equation in $\Z$. 
	We can assume that not all integer coefficients $X,Y,Z,W$ (which itself are some determinants of matrix from \eqref{f:det_A}) vanish because otherwise we obtain a similar equation with a smaller number of variables.
	Without loosing of the generality, assume that $X\neq 0$ and substitute $q_s$ into the identity $q_s p_{s-1} - p_s q_{s-1} = (-1)^s = 1$. 
	We derive
\[
	q_{s-1} p_s X = -p_{s-1} (Y q_{s-1} + Z p_s + W p_{s-1}) -1 
\]
	or, in other words, 
\begin{equation}\label{tmp:divisors}
	(X q_{s-1} + Z p_{s-1})(X p_s + Y p_{s-1}) = YZ p^2_{s-1} - X(W p_{s-1}^2 + 1) := f(p_{s-1}) \,.
\end{equation}
	Fix $p_{s-1} < 2^{-5} p^{1/k}$ and suppose that $f(p_{s-1}) \neq 0$. 
	Then the number of the solutions to equation \eqref{tmp:divisors} can be estimated in terms of the divisor function as $p^{o(1)}$. 
	Further if we know $(q_{s-1},p_{s}, p_{s-1})$, then we determine the matrix from $A$ in at most $M$ ways. 
	Now in the case $f(p_{s-1}) = 0$, we see that there are at most two variants for $p_{s-1}$ and fixing $q_s \le 2^{-5} p^{1/k} $ in  $q_s p_{s-1} - p_s q_{s-1} = 1$, we find the remaining variables in at most $p^{o(1)}$ ways (or just use formula \eqref{tmp:divisors}). 
	Thus we have obtained \eqref{f:sigma_tA_D+}.

	To derive \eqref{f:sigma_A_D+} from \eqref{f:sigma_tA_D+} notice that $A \subseteq \tilde{A}X$, 
	where  $X$ is constructing in an analogues way from $F_M (2^5 (M+1) Q^{1-1/k})$.
	Then, using \eqref{f:sigma_tA_D+}, we get 
\[
	\max_{g,h\in \SL_2 (\F_p)} |A\cap gBh| \le \sum_{x \in X} |\tilde{A}x \cap gBh| \le Mp^{1/k} |X| \ll_M p^{2w_M + \frac{1}{k} (1-2w_M)} 
		\sim |A| \cdot p^{\frac{1-2w_M}{4}} \,.
\] 
This completes the proof of the lemma.
$\hfill\Box$
\end{proof}

	Assume that $|A| \sim p^{2w_M} \gg p^{3/2}$.  
Using formula \eqref{f:growth_in_P+} of Theorem \ref{t:growth_in_P}, as well as Lemma \ref{l:sigma_A}, we obtain an optimal lower bound for $|A^{-1}BA^{-1}|$.
\begin{corollary}	
	Let $w_M > 3/4$. 
	Then 
	\[
	|A^{}BA^{}|,~ |A^{-1}BA^{-1}| 
	\gg 
	p^3 \,.
	\]
\label{c:p^3}
\end{corollary}

Now we are ready to prove Theorem \ref{t:main_MS_new}. 
First of all we obtain the result with the constant equals five and with the exact bounds (on $M$, say) and than subsequently  refine the constant, 
using some additional arguments (which give worse dependence on $M$). 
The method of 
obtaining 
the constant five is more general and can be generalized further, see Theorem \ref{t:CF_alphabet} below and remarks after it. 
One more time, decreasing $C$ in the condition $q=O(p^C)$,   we increasing the constant $\zk$.

Take $n\ge 1$ and consider the equation $a y_1 \dots y_n a' = b$, where $y_j\in Y$, $a,a'\in A$,  $b\in B$ and we will choose the set $Y$ later.  
If this equation has no solutions, then the equation $s y_1 \dots y_n s' = b$, $s\in BA :=S$, $s' \in AB := S'$ has no solutions as well. 
Applying the second part of Lemma \ref{l:sigma_A}, we can estimate the energies $\E(A^{-1},B)$, $\E(B,A)$.
But then formula \eqref{f:energy_CS} gives us
\begin{equation}\label{tmp:11.12_2}
	|S|, |S'| \gg_M |A| p \,.
\end{equation} 
By the arguments as in the proof of Theorem \ref{t:A^n_cap_B}, we obtain (recall that $d_{\min} (\SL_2 (\F_p)) \ge \frac{p-1}{2})$)
\[
	|Y|^n |S| |S'| |B| \ll |\Gr| \left( \frac{|\Gr| |S|}{p} \right)^{1/2} \left( \frac{|\Gr| |S'|}{p} \right)^{1/2} \left( \frac{|\Gr| |Y|}{p} \right)^{n/2}
\]
or, in other words, 
\begin{equation}\label{cond:A,Y}
	|Y|^n |A|^2 \ll p^{2n+4} \,.
\end{equation}
It remains to choose $Y$.
Let $K = |AAA|/ |A|$ and $\tilde{K} = |AA|/|A|$. 
If $\tilde{K} \gg p^6/|A|^{3}$, then $|AA| \gg   p^6/|A|^{2}$ and this is a contradiction with  inequality \eqref{cond:A,Y} 
for 
$Y=AA$ and $n=1$.  
Suppose that $\tilde{K} \ll p^6/|A|^{3}$. 
In \cite[inequality (30)]{MS_Zaremba}, using the Helfgott's method \cite{H}, \cite{RS_SL2}, it was proved that
\[
	|A|^2 p^{-1} \ll_M K \tilde{K} |A| \cdot  K^{2/3} |A|^{1/3} \,,
\]
provided 
\begin{equation}\label{condtmp:A_K}
	|A| \gg p^{3/2} K^{5/2}
\end{equation}
Combining the last estimate with $\tilde{K} \ll p^6/|A|^{3}$, we get 
\[
	K \gg_M \frac{|A|^{11/5}}{p^{21/5}} \,.
\]
It is easy to check, that if \eqref{condtmp:A_K} has no place, then we obtain even better lower bound for $K$. 
Applying inequality \eqref{cond:A,Y} with  $Y=AAA$ and $n=1$ we arrive to a contradiction, provided 
\[
	|A| \sim p^{2w_M} \gg p^{51/26} \,.
\]
In view of \eqref{oop} we can satisfy the last condition taking sufficiently large $M$. 
Thus $A^5 \cap B \neq \emptyset$ and one can calculate the required $M$ by formula \eqref{oop}.

To replace the constant  five in  Theorem \ref{t:main_MS_new} to four  it is enough to show (see inequality \eqref{cond:A,Y}) that $|AA| \gg |A|^{1+c}$, where $c>0$ is an absolute constant and for the last in view of \eqref{f:energy_CS} it is enough to obtain a non--trivial upper bound for the energy of $A$ of the form $\E(A,A) \ll |A|^{3-c}$.  
Suppose for a certain  $T\geq 1$, $\E(A,A) = |A|^3/T$.
By the non--commutative  Balog--Szemer\'edi--Gowers Theorem, see \cite[Theorem 32]{Brendan_rich} or \cite[Proposition 2.43, Corollary 2.46]{TV}
there is $a\in A$ and $A_* \subseteq a^{-1}A$, $|A_*|\gg_T |A|$ such that $|A^3_*| \ll_T |A_*|$.
Here the signs $\ll_T$, $\gg_T$ mean that  all dependences on $T$ are polynomial.
In view of the Helfgott's growth result or Theorem \ref{t:K_lower_SL} it is enough to show that $A_*$ does not belong to a coset of a Borel subgroup. 
But it easily follows from bound \eqref{f:sigma_A_D+} of Lemma \ref{l:sigma_A} (here we assume that $w_M > 1/2$) and the lower bound for size of $A$ (and hence size of $A_*$).

To replace the constant  four in  Theorem \ref{t:main_MS_new} to three notice that the Parseval identity \eqref{f:Parseval_representations} gives us 
$\| A\|^4_o \ll \E(A,A) |\Gr|/p  \ll |A|^{3-c} p^2$. 
Here $c>0$ is an absolute constant, $\| A\|_o = \max_{\rho\neq 1} \| \FF{A}(\rho)\|_o$ and $M$ is taken to be large enough. 
Hence we find a solution to the equation $sas'=b$ provided
\begin{equation}\label{f:using_Fourier_A}
	\frac{|S|^2 |A| |B|}{|\Gr|} \gg p |A|^3 > \frac{|S| |\Gr|}{p}	\| A\|_o \gg  |A| p^3 (|A|^{3-c} p^2)^{1/4} 
\end{equation}
or, in other words,  $|A| \gg p^{\frac{10}{5+c}}$. 
In view of \eqref{oop} we can satisfy the last condition taking sufficiently large $M$.

Finally, we replace the constant  three in  Theorem \ref{t:main_MS_new} to two and further to $1+\eps$.
Let $\Lambda \subset A$ be a set constructing in an analogues way from $F_M (\sqrt{Q})$ but not from $F_M (Q)$.
Clearly, $|\Lambda| \sim p^{w_M} \sim \sqrt{|A|}$ and $\Lambda^2 \subseteq A$.

\begin{lemma}
	Let $X\subseteq B$ be an arbitrary set.
	We have $\E(\Lambda, X) = |\Lambda| |X|$ and $\E(\Lambda^{-1}, X) \le M^4 |\Lambda| |X|$.
	In particular, $|B\Lambda| = |B| |\Lambda|$ and $|\Lambda B| \ge |B| |\Lambda|/M^4$. 
\label{l:E_sqrt}
\end{lemma}
\begin{proof} 
	As in the proof of Lemma \ref{l:sigma_A}, we see that $\Lambda \Lambda^{-1} \in B$ iff $q'_t q_{s-1} \equiv q_s q'_{t-1} \pmod p$ (we use the notation from the lemma). 
	The set $\Lambda$ has been constructed from $F_M (\sqrt{Q})$ and hence we have $q'_t q_{s-1} = q_s q'_{t-1}$. 
	Obviously, $(q_{s-1}, q_s) = (q'_{t-1}, q'_t) = 1$ and hence $q_s = q'_t$, $q_{s-1} = q'_{t-1}$.
	After that we reconstruct the both matrices and obtain    $\E(\Lambda, X) = |\Lambda| |X|$. 

	Similarly, $\Lambda^{-1} \Lambda \in B$ iff $p'_{t-1} q_{s-1} \equiv p_{s-1} q'_{t-1} \pmod p$ and whence $p'_{t-1} q_{s-1} = p_{s-1} q'_{t-1}$.
	Again, $(q_{s-1}, p_{s-1}) = (q'_{t-1}, p'_{t-1}) = 1$ and hence $p_{s-1} = p'_{t-1}$, $q_{s-1} = q'_{t-1}$.
	After that we reconstruct both matrices in at most $M^2$ ways.
	Finally, 
	from 
	\eqref{f:continuants_aj} 
	it follows that the image $\Lambda^{-1} \Lambda$ belongs to a set of cardinality at most $M^2$  
	and whence we obtain    $\E(\Lambda, X) \le M^4 |\Lambda| |X|$. 
This completes the proof of the lemma.
$\hfill\Box$
\end{proof}

After that we redefine $S$ and $S'$ as $B\Lambda$, $\Lambda B$, respectively, and use the calculations from \eqref{f:using_Fourier_A}.
It gives 
\begin{equation}\label{f:using_Fourier_A'}
\frac{|S|^2 |A| |B|}{|\Gr|} \gg p^3 |A|^2 > \frac{|S| |\Gr|}{p}	\| A\|_o \gg  |A|^{1/2} p^4 (|A|^{3-c} p^2)^{1/4} 
\end{equation}
or, in other words,  $|A| \gg p^{\frac{6}{3+c}}$. 
In view of \eqref{oop} we can satisfy the last condition taking sufficiently large $M$. 
Thus we have obtained the integer constant two but it is easy to see that this quantity is, actually, $2-\tilde{c}$, where the absolute constant $\tilde{c}$ depends on $c$. Indeed, just replace $\sqrt{p-1}$ in the definition  
of the set $\Lambda$ 
to 
$p^{(1-\eps)/2}$ for sufficiently small $\eps = \eps(c) >0$ and repeat the calculations above.

In the last step we take an integer parameter $k \sim 1/\epsilon$ and consider $\Lambda_k \subset A$, constructed from $F_M (2^{-1} Q^{1/k})$. 
Let $\tilde{A} = \Lambda^k_k \subset A$ and we have $|\tilde{A}| \sim_k |A|$  
(more precisely, $|A| \ge |\tilde{A}| \ge \eta^k |A|$, where $\eta <1$ is an absolute constant). 
In other words, $A$ and $\tilde{A}$ have comparable sizes. 
In particular, Lemma \ref{l:sigma_A} takes place for $\tilde{A}$, hence $\E(\tilde{A}) \ll_k |\tilde{A}|^{3-c}$ and whence $\| \tilde{A} \|_o \ll_k |\tilde{A}|^{1-c_*}$. 
The set $\tilde{A}$ is the direct product of $k$ copies of  $\Lambda_k$ 
and hence the set of all eigenvalues of the Fourier transform  $\FF{\tilde{A}} (\rho)$ is the $k$th power of the set of all eigenvalues of $\FF{\Lambda}_k (\rho)$.
We will show 
a little bit later 
that this relation, indeed, implies a power saving for the operator norm of $\FF{\Lambda}_k (\rho)$.  
It means that $\|\Lambda_k\|_o \ll_k |\Lambda_k|^{1-c_* (k)}$ for a certain $c_* (k) >0$ and calculations in \eqref{f:using_Fourier_A'} for the equation $s\la_k s' = b$, $\la_k \in \Lambda_k$ give us 
\begin{equation}\label{tmp:10.03.2020_1}
	p^3 |A| |\Lambda_k| \gg |\Lambda_k|^{1-c_* (k)} |A|^{1/2} p^4  \sim |\Lambda_k|^{1-c_* (k)} |\Lambda| p^4 \gg_k |S| \|\Lambda_k\|_o p^2 
\end{equation}
and this is attained for any sufficiently large $M=M(\epsilon)$, $M\gg k/c_* (k)$ 
because inequality \eqref{tmp:10.03.2020_1} is equivalent to 
\begin{equation}\label{tmp:10.03.2020_2}
p^{w_M + 2c_* (k) w_M/k }\gg_k p \,.
\end{equation} 
To demonstrate the required power saving for the operator norm of $\FF{\Lambda}_k (\rho)$ we need to obtain an analogue of Lemma \ref{l:sigma_A} for the set $\Lambda_k$ and $p$ replaced by $p^{1/k}$.  
But because of similarity of $A$ and $\Lambda_k$ the proof is the same  and moreover for $k\ge 4$ we have the following uniform bound $|\Lambda_k \cap gBh| \ll_M p^{1/k}$ for all $g,h\in \SL_2 (\F_p)$ (see the arguments of  the proof of Lemma \ref{l:sigma_A}, in particular, bound \eqref{f:sigma_tA_D+}). 
As for the intersection of $\Lambda_k$ with the dihedral groups $\Gamma$ (another class of maximal subgroups of $\SL_2 (\F_p)$) the arguments are the same again 
(any dihedral subgroup provides even two linear restrictions for the tuple $(p_{s-1}, q_{s-1}, p_s, q_s)$)
and they give the estimate  $|\Lambda_k \cap g \Gamma h| \ll_M p^{1/k}$ for all $g,h\in \SL_2 (\F_p)$ and $k\ge 4$ (the details can be found in \cite{BG} and in \cite[Lemma 21]{NG_S}). 
This completes the proof of Theorem \ref{t:main_MS_new}. 
$\hfill\Box$




\bigskip 

Applying
the second part of Theorem \ref{t:A^n_cap_B_SL2} and the arguments of the proof of the result above (avoid using of  Lemma \ref{l:sigma_A} and Lemma \ref{l:E_sqrt} which appellate to the specific structure of the set $A$), we obtain

\begin{theorem}
	Let $\mathcal{A} \subset \mathbb{N}$ be a finite  set, $|\mathcal{A}| \ge 2$ such that $\mathrm{HD} (F_{\mathcal{A}}) > 1/2 + \delta$, where $\delta >0$. 
	There is an integer constant $C_{\mathcal{A}}(\delta)$ such that for any prime number  $p$ 
	there exist some positive integers $q = O_{\mathcal{A}}(p^{C_{\mathcal{A}}(\delta)})$,  $q\equiv 0 \pmod p$ and  $a$, $(a,q)=1$ 
	having  the property that the ratio $a/q$ has partial quotients  belonging to $\mathcal{A}$.  
\label{t:CF_alphabet}
\end{theorem}


Thanks to \eqref{f:w_2} we see  in particular, that Theorem \ref{t:CF_alphabet} takes place for $\mathcal{A} = \{1,2 \}$. 
Previously, this fact was obtained in \cite{MS_Zaremba} by another approach (although one can check that now our new constant $C_{\mathcal{A}}(\delta)$ is better).  
As the reader can see from the proof, our method is rather general and we do not even need, actually, in restrictions of the form $b_j \in \mathcal{A}$ and 
it is possible to consider other (say, Markov--type) conditions for the partial quotients (of course we still need that the Hausdorff dimension of the corresponding Cantor set is greater than $1/2$).

\bigskip

\noindent{I.D.~Shkredov\\
Steklov Mathematical Institute,\\
ul. Gubkina, 8, Moscow, Russia, 119991}
\\
and
\\
IITP RAS,  \\
Bolshoy Karetny per. 19, Moscow, Russia, 127994\\
and 
\\
MIPT, \\ 
Institutskii per. 9, Dolgoprudnii, Russia, 141701\\
{\tt ilya.shkredov@gmail.com}

\end{document}